\documentclass[12pt,a4paper]{amsart}
\usepackage{amsfonts}
\usepackage{amsthm}
\usepackage{amsmath}
\usepackage{amscd}
\usepackage[latin2]{inputenc}
\usepackage{t1enc}
\usepackage[mathscr]{eucal}
\usepackage{indentfirst}
\usepackage{graphicx}
\usepackage{graphics}
\usepackage{pict2e}
\usepackage{epic}
\numberwithin{equation}{section}
\usepackage[margin=2.9cm]{geometry}
\usepackage{epstopdf} 
\usepackage{mathtools}
\usepackage{tikz}
\usetikzlibrary{positioning}
\usepackage{url}
\usepackage{hyperref}
 
\usepackage{amssymb}
\usepackage{array}
\usepackage{wrapfig}
\usepackage{multirow}
\usepackage{tabu}
\usepackage{MnSymbol}
\usepackage{mathtools}

\DeclarePairedDelimiter\floor{\lfloor}{\rfloor}
\theoremstyle{plain}
\newtheorem{Th}{Theorem}[section]
\newtheorem{Lemma}[Th]{Lemma}

\newtheorem{Prop}[Th]{Proposition}

 \theoremstyle{definition}
\newtheorem{Def}[Th]{Definition}
\newtheorem{Conj}[Th]{Conjecture}
\newtheorem{Rem}[Th]{Remark}
\newtheorem{?}[Th]{Problem}
\newtheorem{Ex}[Th]{Example}

\usepackage[backend=bibtex]{biblatex}
\begin{document}
\title[formulas for the coefficients of the Al-Salam-Chihara polynomials]{Combinatorial formulas for the coefficients of the Al-Salam-Chihara polynomials}

\author{Donghyun Kim}

\email{donghyun\_kim@berkeley.edu}



\begin{abstract} The Al-Salam-Chihara polynomials are an important family of orthogonal polynomials in one variable $x$ depending on 3 parameters $\alpha$, $\beta$ and $q$. They are closely connected to a model from statistical mechanics called the partially asymmetric simple exclusion process (PASEP) and they can be obtained as a specialization of the Askey-Wilson polynomials. We give two different combinatorial formulas for the coefficients of the (transformed) Al-Salam-Chihara polynomials. Our formulas make manifest the fact that the coefficients are polynomials in  $\alpha$, $\beta$ and $q$ with positive coefficients.
\end{abstract}

\maketitle

\section{Introduction}
In the last few decades, there has been a lot of work on finding combinatorial formulas for moments of orthogonal polynomials (see \cite{SMW}, \cite{SW1}, \cite{CW}, \cite{MU}), particularly when they are polynomials with positive coefficients. The Al-Salam-Chihara polynomials are an important class of orthogonal polynomials in one variable $x$ which are connected to a model from statistical mechanics called the partially asymmetric simple exclusion process (PASEP). There have been some works on the combinatorics of the Al-Salam-Chihara polynomials (see \cite{CA}); this work has focused on the moments of the Al-Salam-Chihara polynomials, not the coefficients, as the coefficients fail to be positive polynomials. In this paper, we introduce the \textit{transformed} Al-Salam-Chihara polynomials, which do have positive coefficients and give two manifestly positive combinatorial formulas for the coefficients.

\textit{Orthogonal polynomials} in one variable $(p_n(x))_{n \geq 0}$ are a family of polynomials such that the degree of $p_n(x)$ is $n$ and are orthogonal with respect to a certain measure $\omega$, that is
$$\int p_n(x) p_m(x) d \omega =0, \text{\hspace{2mm}for $m\neq n$ }.$$
Monic orthogonal polynomials can be also defined by a three-term recurrence relation
$$p_{n+1}(x)=(x-b_n)p_n(x)-\lambda_{n}p_{n-1}(x),$$
with $p_0(x)=1$ and $p_{-1}(x)=0$ and where $(b_n)_{n\geq0}$ and $(\lambda_n)_{n\geq1}$ are constants (see \cite{Favard}). We call $(b_n)_{n\geq0}$ and $(\lambda_n)_{n\geq1}$ the \textit{structure constants} of $(p_n(x))_{n \geq 0}$. The \textit{N-th moments} $\mu_N$ of $(p_n(x))_{n \geq 0}$ are defined as 
$\mu_N=\int x^N d \omega,  \text{\hspace{2mm}for $N \geq 0$ }.$

The (monic) Al-Salam-Chihara polynomials are orthogonal polynomials with three free parameters $(a,b,q)$. They are in the basic Askey scheme (see \cite{AS}) and can be obtained as the specialization of the Askey-Wilson polynomials at $c=d=0$. The Al-Salam-Chihara polynomials may be defined by the following three-term recurrence relation (see \cite{ASC})
\begin{align*}
  &p_{n+1}(x)=(x-b_n)p_n(x)-\lambda_{n}p_{n-1}(x) \\
  &b_n=\frac{(a+b)q^n}{2} \nonumber \\
  &\lambda_n=\frac{(1-q^n)(1-a b q^{n-1})}{4}.
\end{align*}

Surprisingly, the moments of the Al-Salam Chihara polynomials are connected to a model from statistical mechanics called the partially asymmetric simple exclusion process (PASEP) (see \cite{SS}, \cite{MU}). The PASEP is a model of interacting particles hopping left and right on a one-dimensional lattice of $N$ sites. Each site can be either occupied by a particle or empty and transition rates between states are proportional to $\alpha$, $\beta$ and $q$ (see Figure \ref{Pic}). Then the partition function $Z_N$ of the PASEP can be written in terms of moments of the Al-Salam-Chihara polynomials (see \cite{MU}, Section 6.1) as follows 
\begin{equation}\label{PASEP}
    Z_N=\sum\limits_{k=0}^{N} \binom{N}{k}(\frac{2 \alpha \beta}{1-q})^N \mu_{N-k},
\end{equation}
using the change of variables
\begin{align}\label{change}
    a=\frac{1-q-\alpha}{\alpha}, \text{\hspace{2mm}} b=\frac{1-q-\beta}{\beta}.
\end{align}
\begin{figure}[ht]\centering
\begin{tikzpicture}[scale=1.5]
\draw[-] (0.5,0) -- (0.8,0);
\draw[-] (0.9,0) -- (1.2,0);
\draw[-] (1.3,0) -- (1.6,0);
\draw[-] (1.7,0) -- (2.0,0);
\draw[-] (2.1,0) -- (2.4,0);
\draw[-] (2.5,0) -- (2.8,0);
\draw[-] (2.9,0) -- (3.2,0);
\filldraw[black] (3.05,0.2) circle (3pt) ;
\filldraw[black] (1.45,0.2) circle (3pt) ;
\draw[->] (3.05, 0.35) .. controls (3.225,0.5) .. (3.4,0.35);
\draw[->] (1.47, 0.35) .. controls (1.65,0.5) .. (1.85,0.35);
\draw[->] (1.43, 0.35) .. controls (1.25,0.5) .. (1.05,0.35);
\draw[->] (0.3, 0.35) .. controls (0.475,0.5) .. (0.65,0.35);
\filldraw[black] (3.225,0.4) circle (0.0000000001pt) node[anchor=south] {$\beta$};
\filldraw[black] (1.65,0.45) circle (0.0000000001pt) node[anchor=south] {1};
\filldraw[black] (1.25,0.45) circle (0.0000000001pt) node[anchor=south] {$q$};
\filldraw[black] (0.475,0.4) circle (0.0000000001pt) node[anchor=south] {$\alpha$};
\end{tikzpicture}
\caption{The figure shows transition rates of the PASEP} \label{Pic}
\end{figure}
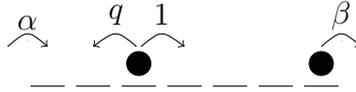

Motivated by the connection \eqref{PASEP} with the PASEP, we consider a family $(p'_n(x))_{n\geq0}$ of polynomials where $$p'_n(x)=(\frac{2\alpha\beta}{q-1})^n p_n(\frac{q-1}{2\alpha \beta}x-1),$$ using the change of variables \eqref{change}. Then the $N$-th moment of $(p'_n(x))_{n\geq0}$ becomes $(-1)^N Z_N$, and we have the following three-term recurrence relation
\begin{align}\label{eq1}
  &p'_{n+1}(x)=(x+b_n)p'_n(x)-\lambda_{n}p'_{n-1}(x) \\
  &b_n=(\alpha+\beta)q^n+2 \alpha \beta [n]_q \nonumber \\
  &\lambda_n=(\alpha\beta)^2[n]_q[n-1]_q +\alpha\beta(\alpha+\beta)q^{n-1}[n]_q+\alpha \beta(q^{2n-1}-q^{n-1}). \nonumber
\end{align} 

In \cite{AL}, a combinatorial formula for $Z_N$ was given in terms of 
 permutation tableaux, showing in particular that it is a polynomial in $\alpha$, $\beta$ and $q$ with positive coefficients. 
 
 Computing $p'_n(x)$ for small $n$ we have 
\begin{align*}
    p'_0(x)=&1\\
    p'_1(x)=&x+(\alpha+\beta)\\
    p'_2(x)=&x^2+( \alpha+\beta+ \alpha q+\beta q+2\alpha\beta)x+(\alpha^2 q+\beta^2 q+\alpha\beta+\alpha\beta q+\alpha\beta^2+\alpha^2\beta)\\
    p'_3(x)=&x^3+(\alpha+\beta+\alpha q+\beta q+ \alpha q^2+\beta q^2+4\alpha\beta+2\alpha\beta q)x^2\\
    &+(\alpha \beta + 3 \alpha^2 \beta + 3 \alpha \beta^2 + 3 \alpha^2 \beta^2 + \alpha^2 q + 2 \alpha \beta q + 3 \alpha^2 \beta q + 
 \beta^2 q + 3 \alpha \beta^2 q + 3 \alpha^2 \beta^2 q \\&+ \alpha^2 q^2 + 2 \alpha \beta q^2 + 3 \alpha^2 \beta q^2 +
  \beta^2 q^2 + 3 \alpha \beta^2 q^2 + \alpha^2 q^3 + \alpha \beta q^3 + \beta^2 q^3)x+(2 \alpha^2 \beta^2 \\&+ \alpha^3 \beta^2 + \alpha^2 \beta^3 + \alpha^2 \beta q + \alpha^3 \beta q + \alpha \beta^2 q + 
 2 \alpha^2 \beta^2 q + \alpha^3 \beta^2 q + \alpha \beta^3 q + \alpha^2 \beta^3 q + \alpha^2 \beta q^2 \\&+ 
 2 \alpha^3 \beta q^2 + \alpha \beta^2 q^2 + 2 \alpha^2 \beta^2 q^2 + 2 \alpha \beta^3 q^2 + \alpha^3 q^3 + 
 \alpha^2 \beta q^3 + \alpha \beta^2 q^3 + \beta^3 q^3).
\end{align*}
Note that it is not obvious from the recurrence \eqref{eq1} that the coefficients are polynomials in $\alpha$, $\beta$ and $q$ with positive coefficients. 

It is worth noting that specializing to $q=1$, the polynomial $p'_n(x)$ becomes 
\begin{equation*}
    p'_n(x)=\sum\limits_{k=0}^{n} (\binom{n}{k}\prod\limits_{i=n-k}^{n-1}(\alpha+\beta+i \alpha\beta))x^{n-k},
\end{equation*}
which can be easily proved by induction. The coefficients of $p'_n(x)$ have a nice factorization formula in this case; however they do not factorize in general.

In this paper we will give two different combinatorial formulas for these coefficients making manifest that they are polynomials in $\alpha$, $\beta$ and $q$ with positive coefficients. To do this we introduce the following more general orthogonal polynomials.
\begin{Def}\label{definition} The \textit{transformed Al-Salam-Chihara} polynomials $(\hat{p}_n(x))_{n\geq0}$ are the family of orthogonal polynomials in one variable $x$ depending on parameters $\alpha$, $\beta$, $\epsilon_1$, $\epsilon_2$ and $q$ defined  by the following three-term recurrence relation  
\begin{align}\label{eq3}
  &\hat{p}_{n+1}(x)=(x+b_n)\hat{p}_n(x)-\lambda_{n}\hat{p}_{n-1}(x) \\
  &b_n=(\alpha+\beta)q^n+(\epsilon_1+\epsilon_2)[n]_q \nonumber\\
  &\lambda_n=\epsilon_1\epsilon_2 [n]_q [n-1]_q+(\alpha \epsilon_2+\beta \epsilon_1)q^{n-1} [n]_q+\alpha \beta (q^{2n-1}-q^{n-1}). \nonumber
\end{align}
\end{Def}
\begin{Rem}
The connection with the PASEP provided some inspiration for Definition \ref{definition}. In particular, there is a 1-parameter generalization of the partition function $Z_N$ called the fugacity partition function $Z_N(\xi)$ where $\xi$ is a variable keeping track of the number of particles for each state. This connection leads to the following family of orthogonal polynomials defined by the three-term recurrence relation \begin{align}\label{eq2}
  &p''_{n+1}(x)=(x+b_n)p''_n(x)-\lambda_{n}p''_{n-1}(x) \\
  &b_n=(\xi\alpha+\beta)q^n+(1+\xi) \alpha \beta [n]_q \nonumber \\
  &\lambda_n=\xi\alpha\beta)^2[n]_q[n-1]_q +\xi\alpha\beta(\alpha+\beta)q^{n-1}[n]_q+\xi\alpha \beta(q^{2n-1}-q^{n-1}), \nonumber
\end{align} 
which is a $\xi$-analogue of \eqref{eq1} (see \cite{xi}). We can recover \eqref{eq2} from the more general setting of \eqref{eq3} by plugging $\alpha\rightarrow \xi\alpha$, $\epsilon_1\rightarrow \xi \alpha \beta$ and $\epsilon_2\rightarrow \alpha \beta$.
\end{Rem}
\begin{Rem}
The coefficients $[x^n]p_{n+k}(x)$ of the Al-Salam-Chihara polynomials and the coefficients $[x^n]\hat{p}_{n+k}(x)$ of the transformed Al-Salam-Chihara polynomials are connected as follows
\begin{equation*}
    [x^n]p_{n+k}(x)=\sum\limits_{i=0}^{k}\binom{n+i}{n}(\frac{q-1}{2\alpha\beta})^{k-i}([x^{n+i}]\hat{p}_{n+k}(x)),
\end{equation*}
where $\epsilon_1=\epsilon_2=\alpha \beta$,  $a=\frac{1-q-\alpha}{\alpha}$ and $b=\frac{1-q-\beta}{\beta}.$
\end{Rem}
  The  rest of the paper studies the transformed Al-Salam-Chihara polynomials from Definition \ref{definition}. We give two formulas for the coefficient $g_{n+k,n}$ of $x^n$ in $\hat{p}_{n+k}(x)$. Our two formulas represent $g_{n+k,n}$ as polynomials in $X_i=\alpha q^i+\epsilon_1 [i]_q$ and $Y_i=\beta q^i+\epsilon_2 [i]_q$ (see \eqref{eq5}) where the coefficients lie in $\mathbb{Z}[q]$. For example, by our first result (Theorem \ref{main1}) we have
\begin{align}\label{g311}
    g_{3,1}=(\sum_{0\leq i<j\leq 2}X_i X_j)+(X_0Y_1+Y_0X_2+X_1Y_2+\binom{3}{1}_qX_0Y_0)+(\sum_{0\leq i<j\leq 2}Y_i Y_j), 
\end{align}
and by our second result (Theorem \ref{main2}) we have
\begin{align}\label{g312}
    g_{3,1}=(\sum_{0\leq i<j\leq 2}X_i X_j)+(X_0Y_0+X_0Y_1+q^2 X_0Y_0+X_1Y_0+X_1Y_1+q X_1Y_1)+(\sum_{0\leq i<j\leq 2}Y_i Y_j).
\end{align}

Note that \eqref{g312} is invariant as a polynomial in $X_i$'s and $Y_i$'s under the exchange $X_i \leftrightarrow Y_i$. This is the case in general for our second formula and will be explained in Remark \ref{sym}. The first formula, however, is not invariant as a polynomial in $X_i$'s and $Y_i$'s under the exchange $X_i \leftrightarrow Y_i$ as one can see from \eqref{g311}. So far, it is not clear how these two formulas are connected.

The structure of this paper is as follows. In Section 2, we will state the main results of this paper with examples. In Section 3, we will prove our first result (Theorem \ref{main1}). In Section 4, we will prove our second result (Theorem \ref{main2}). In Section 5, we will prove Theorem \ref{2pos} which is a partial result of the conjecture regarding the minors of the matrix of coefficients $G=(g_{n,i})_{n,i}$. 

\textbf{Acknowledgments:} The author is thankful to his advisor Lauren Williams for mentorship, valuable comments, and helping me revise the draft. The author would also like to thank Sylvie Corteel for her helpful comments and explanations.  
\section{Main Results}
This section states the main results of this paper with examples. Throughout this section, we set 
\begin{equation}\label{eq5}
    X_i=\alpha q^i+\epsilon_1 [i]_q, \hspace{2mm} Y_i=\beta q^i+\epsilon_2 [i]_q
\end{equation}
for $i \geq 0$.
\subsection{The first formula for the coefficients of \texorpdfstring{$\hat{p}_n(x)$}{Lg}}
\begin{Def} \label{deff}
Define a sequence $Z_n$ for $n\geq0$ by
\begin{equation*}
Z_{n}=\begin{cases*}
X_{\floor{\frac{n}{2}}} & if $n$ is even\\
Y_{\floor{\frac{n}{2}}} & if $n$ is odd
\end{cases*}.
\end{equation*}
For a partition (weakly decreasing sequence of non-negative integers) $\mu=(\mu_1,\cdots,\mu_l)$, we define
\begin{equation*}
s_m(\mu)=\begin{cases*}
min(|\{i | \mu_i=0\}|,|\{i | \mu_i=2m+1\}|) & if $\mu_1=2m+1$\\
0 & otherwise
\end{cases*}.
\end{equation*}
Denoting $k=s_m(\mu)$, we define a weight $u_m (\mu)$ to be

$$u_m (\mu)=(\prod_{i=1}^{l-k}Z_{\mu_{l+1-i}+2(i-1)})(\binom{m+l}{k}_q Y_0\cdots Y_{k-1}).$$
\end{Def}
\begin{Ex} For $m=1$, consider a partition $\mu=(3,3,1,0)$. Then $s_1((3,3,1,0))=min(|\{i | \mu_i=0\}|=1,|\{i | \mu_i=2m+1\}|=2)=1$, so we have

$$u_1((3,3,1,0))=Z_0 Z_{1+2} Z_{3+4} \binom{5}{1}_q Y_0=X_0 Y_1 Y_3  \binom{5}{1}_q Y_0.$$
For a partition $\mu=(3,3,0,0)$, we have $s_1((3,3,0,0))=min(|\{i | \mu_i=0\}|=2,|\{i | \mu_i=2m+1\}|=2)=2$, so this gives

$$u_1((3,3,0,0)=Z_0 Z_{0+2} \binom{5}{2}_q Y_0 Y_1=X_0 X_1  \binom{5}{2}_q Y_0Y_1.$$
\end{Ex}
\begin{Th} \label{main1}
The coefficient $g_{n+k,n}$ of $x^n$ in $\hat{p}_{n+k}(x)$ is given by 
\begin{equation*}
    g_{n+k,n}=\sum_{\mu \subseteq (k) \times (2n+1)} u_n (\mu),
\end{equation*}
ie. it is the weighted sum over all Young diagrams contained in a $(k)\times(2n+1)$ rectangle, where the weight is given by Definition \ref{deff}.
\end{Th}
\begin{Ex}
By Theorem \ref{main1}, we have
$$g_{k,0}=\sum_{\mu \subseteq (k) \times (1)} u_0 (\mu)=\sum_{i=0}^{k} u_0 ((1^{k-i},0^{i})).$$ 
If $k-i \leq i$, then  $$u_0 ((1^{k-i},0^{i}))=X_0\cdots X_{i-1} \binom{k}{k-i}_q Y_0\cdots Y_{k-i-1}=X_0\cdots X_{i-1} \binom{k}{i}_q Y_0\cdots Y_{k-i-1}.$$
If $k-i > i$, then
 $$u_0 ((1^{k-i},0^{i}))=X_0\cdots X_{i-1} Y_i \cdots Y_{k-i-1}\binom{k}{i}_q Y_0\cdots Y_{i-1}=X_0\cdots X_{i-1} \binom{k}{i}_q Y_0\cdots Y_{k-i-1}.$$
 In both cases, we have $u_0 ((1^{k-i},0^{i}))=X_0\cdots X_{i-1} \binom{k}{i}_q Y_0\cdots Y_{k-i-1}.$ Thus we have
 \begin{equation}\label{0}
     g_{k,0}=\sum_{i=0}^{k} \binom{k}{i}_q  X_0\cdots X_{i-1} Y_0\cdots Y_{k-i-1}.
 \end{equation}
\end{Ex}
\begin{Ex}
By Theorem \ref{main1}, we have
\begin{align*}
    g_{3,1}=&\sum_{\mu \subseteq (2) \times (3)} u_1 (\mu)\\=&u_1 ((0,0))+u_1 ((1,0))+u_1 ((2,0))+u_1 ((3,0))+u_1 ((1,1))\\&+u_1 ((2,1))+u_1 ((3,1))+u_1 ((2,2))+u_1 ((3,2))+u_1 ((3,3))\\=&X_0X_1+X_0Y_1+X_0X_2+\binom{3}{1}_qX_0Y_0+Y_0Y_1 \\&+Y_0X_2+Y_0Y_2+X_1X_2+X_1Y_2+Y_1Y_2
    \\=&(\sum_{0\leq i<j\leq 2}X_i X_j)+(X_0Y_1+Y_0X_2+X_1Y_2+\binom{3}{1}_qX_0Y_0)+(\sum_{0\leq i<j\leq 2}Y_i Y_j).
\end{align*}
\end{Ex}
\subsection{The second formula for the coefficients of \texorpdfstring{$\hat{p}_n(x)$}{Lg}}
\begin{Def}\label{S}
For a set $S$ with integer elements, we define $S(k)$ to be the $k$-th smallest element of the set $(\{0\}\cup \mathbb{N})-S$. For example, when $S=\{1,4,7\}$, we have $S(1)=0$, $S(2)=2$, $S(3)=3$ and $S(4)=5$. 
We also define $\lambda_S$ to be $(i_1,i_2-1,\cdots ,i_s-s+1)$ where  $S=\{ i_1<\cdots<i_s\}$.
\end{Def}
\begin{Def}\label{weight}
For a set $A=\{i_1<\cdots<i_a\} \subseteq \{0,\cdots,n+a-1\}$ and a set $B \subseteq \{0,\cdots,n+a+b-1\}$ with $|B|=b$, denoting $B \cap \{n+b,\cdots,n+b+a-1\}=\{n+b+a-j_k<\cdots<n+b+a-j_1\}$ and $\mu=\lambda_A=(i_1,i_2-1,\cdots ,i_a-a+1)$, we define a weight $w_n(A,B)$ to be
$$w_n(A,B)=(\prod_{i \in A} X_i)(\prod_{i \in B\cap\{0,\cdots,n+b-1\}} Y_i)(\prod_{l=1}^{k}(q^{(n+b+a-j_l)-B(\mu_{j_l}+l)}Y_{B(\mu_{j_l}+l)})) .$$
\end{Def}
Definition \ref{weight} is motivated by the bijective proof of the simple identity 
\begin{equation}\label{simple}q^{\binom{a}{2}}\binom{n+a}{a}_q q^{\binom{b}{2}}\binom{n+a+b}{b}_q=
q^{\binom{a}{2}}\binom{n+a+b}{a}_q q^{\binom{b}{2}}\binom{n+b}{b}_q
\end{equation} which will be given in Section \ref{bijection}.

\begin{Th} \label{main2}
The coefficient $g_{n+k,n}$ of $x^n$ in $\hat{p}_{n+k}(x)$ is given by
$$g_{n+k,n}=\sum_{a+b=k}(\sum_{\substack{A \subseteq \{0,\cdots,(n+a-1)\}\\|A|=a}}(\sum_{\substack{B \subseteq \{0,\cdots,(n+a+b-1)\}\\|B|=b}}w_n(A,B))),$$
ie. it is the weighted sum over all pairs $(A,B)$ such that $A \subseteq \{0,\cdots,(n+a-1)\}$ with $|A|=a$ and $B \subseteq \{0,\cdots,(n+a+b-1)\}$ with $|B|=b$, where the weight is given by Definition \ref{weight}.
\end{Th}
\begin{Ex}\label{ex38}
We compute $w_0(A,B)$ as follows. Since $n=0$ there is only one possible choice for $A$ which is $\{0,\cdots,a-1\}$, so  $\mu=\lambda_A=(0,\cdots,0)$. Suppose there are $k$ elements in a set $B \cap \{b,\cdots,b+a-1\}$ then these elements will change to elements in $(\{0,\cdots,b-1\}-B)$. So we have $w_0(A,B)=X_0\cdots X_{a-1}(q^{\sum B-\binom{b}{2}})Y_0\cdots Y_{b-1}.$ Thus we have
\begin{align*}
    g_{k,0}=&\sum_{a+b=k}(\sum_{\substack{B \subseteq \{0,\cdots,a+b-1\}\\|B|=b}}X_0\cdots X_{a-1}(q^{\sum B-\binom{b}{2}})Y_0\cdots Y_{b-1}) \\=&
    \sum_{a+b=k}\binom{a+b}{b}_q X_0\cdots X_{a-1}Y_0\cdots Y_{b-1}.
\end{align*}
In this case the formula is identical to \eqref{0}.
\end{Ex}
\begin{Ex} By Theorem \ref{main2}, we have
\begin{align*}
    g_{3,1}=&\sum_{\substack{B\subseteq\{0,1,2\}\\|B|=2}}w_1(\phi,B)+\sum_{\substack{A\subseteq\{0,1\}\\|A|=1}}(\sum_{\substack{B\subseteq\{0,1,2\}\\|B|=1}}w_1(A,B))+\sum_{\substack{A\subseteq\{0,1,2\}\\|A|=2}}w_1(A,\phi)\\=&(\sum_{0\leq i<j\leq 2}X_i X_j)+(X_0Y_0+X_0Y_1+q^2 X_0Y_0+X_1Y_0+X_1Y_1+q X_1Y_1)+(\sum_{0\leq i<j\leq 2}Y_i Y_j).
\end{align*}
\end{Ex}

On the way to prove Theorem \ref{main2}, we introduce the following, which extends $q$-binomial coefficient. 
\begin{Def}\label{wei}
For a weakly increasing composition $\mu=(\mu_1,\cdots,\mu_a)$ such that $-1\leq\mu_1,\cdots,\mu_a\leq n$ and a set $B \subseteq \{0,\cdots,n+a+b-1\}$ with $|B|=b$, denoting $B \cap \{n+b,\cdots,n+b+a-1\}=\{n+b+a-j_k<\cdots<n+b+a-j_1\}$, we define  a weight $m^{\mu}_n(B)$ to be
$$m^{\mu}_n(B)=(\prod_{i \in B\cap\{0,\cdots,n+b-1\}} Y_i)(\prod_{l=1}^{k}(q^{(n+b+a-j_l)-B(\mu_{j_l}+l)}Y_{B(\mu_{j_l}+l)}) ,$$
where we define $B(0)=-1$ and $Y_{-1}=q^{-1}(\beta-\epsilon_2)$. We also define a \textit{generalized q-binomial coefficient} $M^{\mu}_n(b)$ to be
$$M^{\mu}_n(b)=\sum_{\substack{B \subseteq \{0,\cdots,n+a+b-1\}\\|B|=b}} m^{\mu}_n(B).$$
\end{Def}
We named $M^{\mu}_n(b)$ a generalized $q$-binomial coefficient because  when $\epsilon_2=0$, we have $M^{\mu}_n(b)=q^{\binom{b}{2}}\binom{n+a+b}{b}_q (\beta)^b$, where $a$ is a number of components in $\mu=(\mu_1,\cdots,\mu_a)$. Note that $q$-binomial coefficients have the following well known identities
\begin{equation}\label{zzzz}\binom{n+a+b+1}{b}_q=q^{n+a+1}\binom{n+a+b}{b-1}_q+\binom{n+a+b}{b}_q\end{equation}
\begin{equation}\label{q1q1}
    [n+a+b+1]_q\binom{n+a+b}{b}_q=[n+a+1]_q\binom{n+a+b+1}{b}_q.
\end{equation}
We will give a generalization of \eqref{zzzz} in Lemma \ref{ana} and a generalization of \eqref{q1q1} in Lemma \ref{t}. These two lemmas will be key ingredients for the proof of Theorem \ref{main2}.
\subsection{Positivity of minors of the matrix of coefficients}
Motivated by \cite{CW} (Conjecture 4.4), we make the following conjecture.
\begin{Conj}\label{conj}
Let $G=(g_{n,i})_{n,i}$ be the infinite array of coefficients $g_{n,i}=[x^i]\hat{p}_n(x)$ where $n,i \in \mathbb{Z}_{\geq0}$ and $g_{n,i}=0$ if $i>n$. Then the (non-vanishing) minors of $G$ are polynomials with positive coefficients. 
\end{Conj}
Specializing $\alpha \rightarrow \xi \alpha$, $\epsilon_1 \rightarrow \xi \alpha \beta$ and $\epsilon_2 \rightarrow \alpha \beta$, Conjecture \ref{conj} recovers the positivity conjecture for Koornwinder moments when $\gamma=\delta=0$ (see \cite{CW}, Conjecture 4.4).
\begin{Prop}
Conjecture \ref{conj} is true for the following cases.

(1) $\alpha=0$ (or $\beta=0$)

(2) $\alpha=\epsilon_1$ (or $\beta=\epsilon_2$)

(3) $\epsilon_1=0$ (or $\epsilon_2=0$)
\end{Prop}
\begin{proof}
   By Remark \ref{pos1} and Remark \ref{pos2} together with the Lindstr\"{o}m-Gessel-Viennot lemma (see \cite{GV}) proves the proposition.
\end{proof}
For polynomials $f_1$ and $f_2$ we will write $f_1\succeq f_2$ if $(f_1-f_2)$ is a polynomial with positive coefficients. The following theorem shows that 2 by 2 minors of $G=(g_{n,i})_{n,i}$ having $g_{n,n}=1$ as a lower left entry are polynomials with positive coefficients. The proof will use Theorem \ref{main2}.
\begin{Th} \label{2pos}
For non-negative integers $n,a$ and $b$, we have
\begin{equation*}
    g_{n+a+b,n+a}g_{n+a,n}\succeq g_{n+a+b,n}.
\end{equation*}
\end{Th}
\begin{Ex}
We have 
\begin{align*}
    g_{4,2}g_{2,1}-g_{4,1}=&\alpha^2 \beta + \alpha \beta^2 + 2 \alpha^2 \epsilon_1 + 
 2 \alpha \beta \epsilon_1 + \beta^2 \epsilon_1 + \beta \epsilon_1^2 + 
 \alpha^2 \epsilon_2 + 2 \alpha \beta \epsilon_2 + 2 \beta^2 \epsilon_2 + 
 \alpha \epsilon_1 \epsilon_2 \\&+ \beta \epsilon_1 \epsilon_2 + \alpha \epsilon_2^2 +
  \alpha^3 q + 2 \alpha^2 \beta q + 2 \alpha \beta^2 q + \beta^3 q + 
 \alpha^2 \epsilon_1 q + 2 \alpha \beta \epsilon_1 q + \beta^2 \epsilon_1 q \\&+ 
 \beta \epsilon_1^2 q + \alpha^2 \epsilon_2 q + 2 \alpha \beta \epsilon_2 q + 
 \beta^2 \epsilon_2 q + \alpha \epsilon_1 \epsilon_2 q + 
 \beta \epsilon_1 \epsilon_2 q + \alpha \epsilon_2^2 q + \alpha^3 q^2 + 
 2 \alpha^2 \beta q^2 \\&+ 2 \alpha \beta^2 q^2 + \beta^3 q^2 + 
 2 \alpha \beta \epsilon_1 q^2 + \beta^2 \epsilon_1 q^2 + 
 \alpha^2 \epsilon_2 q^2 + 2 \alpha \beta \epsilon_2 q^2 + \alpha^2 \beta q^3 +
  \alpha \beta^2 q^3.
\end{align*}
We conclude $g_{4,2}g_{2,1}\succeq g_{4,1}$.
\end{Ex}

\section{Proof of Theorem \ref{main1}}
Our goal in this section is to prove Theorem \ref{main1}, which gives a combinatorial formula for the coefficients of the transformed Al-Salam-Chihara polynomials as a weighted sum over Young diagrams contained in a rectangle. One step along the way is to prove Proposition \ref{prop}, which gives an analogous result in a simpler setting. 
\subsection{A motivating result}
Consider the family of orthogonal polynomials $(\tilde{p}_n (x))_{n\geq 0}$ in one variable $x$, defined by the following three-term recurrence relation
\begin{align*}
&\tilde{p}_{n+1}(x)=(x+\tilde{b}_n)\tilde{p}_n(x)-\tilde{\lambda}_{n}\tilde{p}_{n-1}(x)\\&\tilde{b}_n=X_n+Y_n\\ &\tilde{\lambda}_n=Y_{n-1} X_{n},
\end{align*}
where $X_i$'s and $Y_i$'s are indeterminates.

\begin{Def} \label{defw}
For a partition $\mu=(\mu_1,\mu_2,\cdots,\mu_l)$, we define a weight $w(\mu)$ to be $$w(\mu)=\prod\limits_{i=1}^{l} Z_{\mu_{l+1-i}+2(i-1)}=Z_{\mu_l}Z_{\mu_{l-1}+2}Z_{\mu_{l-2}+4}\cdots Z_{\mu_1+2(l-1)},$$
where $Z_n$ is given by Definition \ref{deff}.
\end{Def}

\begin{Prop} \label{prop} The coefficient $\tilde{g}_{n+k,n}$ of $x^n$ in $\tilde{p}_{n+k}(x)$ is given by 
$$\tilde{g}_{n+k,n}=\sum_{\mu \subseteq (k) \times (2n+1)} w(\mu),$$
ie. it is the weighted sum over all Young diagrams contained in a $(k)\times(2n+1)$ rectangle, where the weight is given by Definition \ref{defw}.
\end{Prop}
\begin{proof}
The proof goes with the induction. We first prove $\tilde{g}_{k,0}$ is given by the above formula. The base cases $\tilde{g}_{0,0}=1$ and $\tilde{g}_{1,0}=X_0+Y_0$ are trivially satisfied. It suffices to prove that the formula satisfies the recurrence $\tilde{g}_{k+1,0}=\tilde{b}_{k} \tilde{g}_{k,0}-\tilde{\lambda}_{k} \tilde{g}_{k-1,0}$. Since $\sum\limits_{\mu \subseteq k \times 1} w(\mu)=\sum\limits_{i=0}^{k} X_0\cdots X_{i-1} Y_{i} \cdots Y_{k-1}$, we have 
\begin{align*}
 \sum\limits_{\mu \subseteq (k+1) \times 1} w(\mu)=&\sum\limits_{i=0}^{k+1} X_0\cdots X_{i-1} Y_{i} \cdots Y_{k}=X_0 \cdots X_k+(\sum\limits_{i=0}^{k} X_0\cdots X_{i-1} Y_{i} \cdots Y_{k-1}) Y_k \\
 =&\left(\sum\limits_{i=0}^{k} X_0\cdots X_{i-1} Y_{i} \cdots Y_{k-1}-\sum\limits_{i=0}^{k-1} X_0\cdots X_{i-1} Y_{i} \cdots  Y_{k-1}\right) X_k \\&+(\sum\limits_{i=0}^{k} X_0\cdots X_{i-1} Y_{i} \cdots Y_{k-1}) Y_k \\
 =&(\sum\limits_{i=0}^{k} X_0\cdots X_{i-1} Y_{i} \cdots Y_{k-1})(X_k+Y_k)-(\sum\limits_{i=0}^{k-1} X_0\cdots X_{i-1} Y_{i} \cdots Y_{k-2})Y_{k-1}X_k \\
 =& \tilde{b}_k (\sum\limits_{\mu \subseteq (k) \times 1} w(\mu))- \tilde{\lambda}_k (\sum\limits_{\mu \subseteq (k-1) \times 1} w(\mu)).
\end{align*}
Now we will show that for $n>0$, the formula satisfies the recurrence $$\tilde{g}_{n+k+1,n}=\tilde{g}_{n+k,n-1}+\tilde{b}_{n+k} (\tilde{g}_{n+k,n})-\tilde{\lambda}_{n+k} (\tilde{g}_{n+k-1,n})$$ which is checked as follows

\begin{align*}
    &\sum_{\mu \subseteq (k+1) \times (2n+1)} w(\mu)\\=&\sum_{\mu \subseteq (k+1) \times (2n-1)} w(\mu)+\sum_{\substack{\mu \subseteq (k+1) \times (2n+1) \\ \mu_1=2n}} w(\mu)+\sum_{\substack{\mu \subseteq (k+1) \times (2n+1) \\ \mu_1=2n+1}} w(\mu) \\
    =&\sum_{\mu \subseteq (k+1) \times (2n-1)} w(\mu)+\sum_{\mu' \subseteq (k) \times {2n}} w(\mu') X_{n+k}+\sum_{\mu' \subseteq (k) \times {2n+1}} w(\mu') Y_{n+k} \\
    =&\sum_{\mu \subseteq (k+1) \times (2n-1)} w(\mu)+\left(\sum_{\mu' \subseteq (k) \times (2n+1)} w(\mu')-\sum_{\substack{\mu' \subseteq (k) \times (2n+1) \\ \mu'_1=2n+1}} w(\mu')\right) X_{n+k}\\&+\sum_{\mu' \subseteq (k) \times (2n+1)} w(\mu') Y_{n+k}
    \\
    =&\sum_{\mu \subseteq (k+1) \times (2n-1)} w(\mu)+(\sum_{\mu' \subseteq (k) \times (2n+1)} w(\mu'))(X_{n+k}+Y_{n+k})-\sum_{\substack{\mu' \subseteq (k) \times (2n+1) \\ \mu'_1=2n+1}} w(\mu') X_{n+k}   \\
    =&\sum_{\mu \subseteq (k+1) \times (2n-1)} w(\mu)+\tilde{b}_{n+k}\left(\sum_{\mu' \subseteq (k) \times (2n+1)} w(\mu')\right)-\tilde{\lambda}_{n+k}\left(\sum_{\mu'' \subseteq (k-1) \times (2n+1) } w(\mu'')\right). 
\end{align*}

\end{proof}
\begin{Rem} \label{pos1}
Consider a vertex set $V=\{(i,j) \in \mathbb{Z} \times \mathbb{Z}_{\geq 0} \mid -2j-1 \leq i \leq 0\}$ and directed edges connecting every horizontally or vertically adjacent pair of vertices in an increasing direction (see Figure \ref{graph1}). We will assign weights to directed edges as follows. For a directed edge of the form $(i,j) \rightarrow (i+1,j)$, we give a weight 1 and for a directed edge of the form $(i,j) \rightarrow (i,j+1)$, we give a weight $Z_{i+2j+1}$ (defined in Definition \ref{deff}). The weight $W(P)$ of a path $P$ is defined to be the product of the weights of its edges. Set $u_i=(-2i-1,i)$ and $v_i=(0,i)$ for $i\geq0$. Then the formula for $\tilde{g}_{n+k,n}$ given in Proposition \ref{prop} is equivalent to $\tilde{g}_{n+k,n}=\sum\limits_{P: u_{n}\rightarrow v_{n+k}} W(P)$, where the sum is over all paths $P$ from $u_n$ to $v_{n+k}$.
\end{Rem}
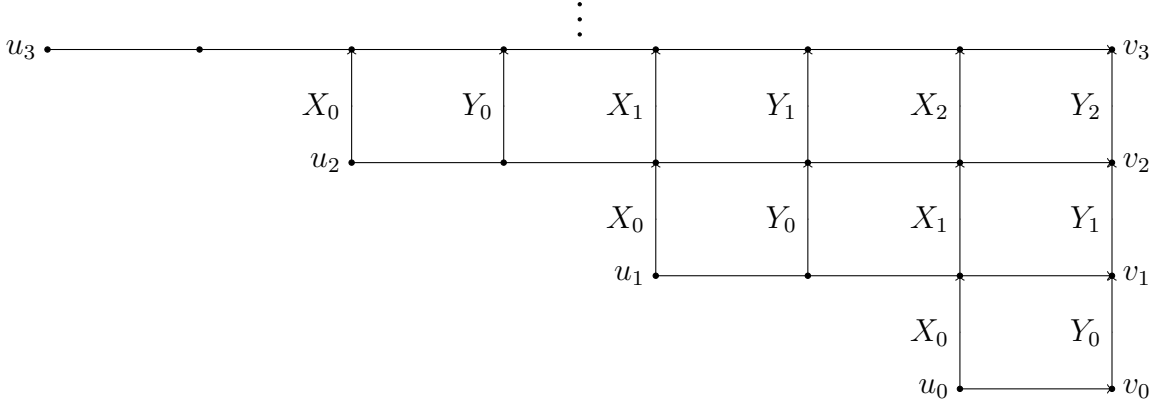
\begin{figure}[ht] 
\begin{tikzpicture}
\filldraw[black] (5,0.2)circle (0.6pt);
\filldraw[black] (5,0.4)circle (0.6pt);
\filldraw[black] (5,0.6)circle (0.6pt);

\draw[->] (-2,0) -- (12,0);
\draw[->] (2,-1.5) -- (12,-1.5);
\draw[->] (6,-3) -- (12,-3);
\draw[->] (10,-4.5) -- (12,-4.5);

\draw[->] (2,-1.5) -- (2,-0);
\filldraw[black] (2,-0.75)circle (0.0000001pt) node[anchor=east] {$X_0$};
\draw[->] (4,-1.5) -- (4,-0);
\filldraw[black] (4,-0.75)circle (0.0000001pt) node[anchor=east] {$Y_0$};
\draw[->] (6,-1.5) -- (6,-0);
\filldraw[black] (6,-0.75)circle (0.0000001pt) node[anchor=east] {$X_1$};
\draw[->] (8,-1.5) -- (8,-0);
\filldraw[black] (8,-0.75)circle (0.0000001pt) node[anchor=east] {$Y_1$};
\draw[->] (10,-1.5) -- (10,-0);
\filldraw[black] (10,-0.75)circle (0.0000001pt) node[anchor=east] {$X_2$};
\draw[->] (12,-1.5) -- (12,-0);
\filldraw[black] (12,-0.75)circle (0.0000001pt) node[anchor=east] {$Y_2$};

\draw[->] (6,-3) -- (6,-1.5);
\filldraw[black] (6,-0.75-1.5)circle (0.0000001pt) node[anchor=east] {$X_0$};
\draw[->] (8,-1.5-1.5) -- (8,-0-1.5);
\filldraw[black] (8,-0.75-1.5)circle (0.0000001pt) node[anchor=east] {$Y_0$};
\draw[->] (10,-1.5-1.5) -- (10,-0-1.5);
\filldraw[black] (10,-0.75-1.5)circle (0.0000001pt) node[anchor=east] {$X_1$};
\draw[->] (12,-1.5-1.5) -- (12,-0-1.5);
\filldraw[black] (12,-0.75-1.5)circle (0.0000001pt) node[anchor=east] {$Y_1$};

\draw[->] (10,-1.5-1.5-1.5) -- (10,-0-1.5-1.5);
\filldraw[black] (10,-0.75-1.5-1.5)circle (0.0000001pt) node[anchor=east] {$X_0$};
\draw[->] (12,-1.5-1.5-1.5) -- (12,-0-1.5-1.5);
\filldraw[black] (12,-0.75-1.5-1.5)circle (0.0000001pt) node[anchor=east] {$Y_0$};

\filldraw[black] (-2,0)circle (1pt) node[anchor=east] {$u_3$};
\filldraw[black] (0,0)circle (1pt);
\filldraw[black] (2,0)circle (1pt);
\filldraw[black] (4,0)circle (1pt);

\filldraw[black] (6,0)circle (1pt);

\filldraw[black] (8,0)circle (1pt);

\filldraw[black] (10,0)circle (1pt);

\filldraw[black] (12,0)circle (1pt) node[anchor=west] {$v_3$};

\filldraw[black] (2,0-1.5)circle (1pt) node[anchor=east] {$u_2$};

\filldraw[black] (4,0-1.5)circle (1pt);

\filldraw[black] (6,-1.5)circle (1pt);

\filldraw[black] (8,0-1.5)circle (1pt);

\filldraw[black] (10,0-1.5)circle (1pt);

\filldraw[black] (12,0-1.5)circle (1pt) node[anchor=west] {$v_2$};

\filldraw[black] (6,-3)circle (1pt) node[anchor=east] {$u_1$};

\filldraw[black] (8,0-3)circle (1pt);

\filldraw[black] (10,0-3)circle (1pt);

\filldraw[black] (12,0-3)circle (1pt) node[anchor=west] {$v_1$};

\filldraw[black] (10,0-4.5)circle (1pt) node[anchor=east] {$u_0$};

\filldraw[black] (12,0-4.5)circle (1pt) node[anchor=west] {$v_0$};

\end{tikzpicture}\caption{The figure shows the weighted directed graph constructed in Remark \ref{pos1}.}\label{graph1}
\end{figure}

From now on, we specify $X_i$'s and $Y_i$'s as given in \eqref{eq5}
\begin{equation*}
    X_i=\alpha q^i+\epsilon_1 [i]_q, \hspace{2mm} Y_i=\beta q^i+\epsilon_2 [i]_q.
\end{equation*}
Then the structure constants for the transformed Al-Salam-Chihara polynomials can be represented as follows
$$b_n=X_n+Y_n$$
$$\lambda_n=Y_{n-1}X_n-\alpha(\beta-\epsilon_2)q^{n-1}.$$ 
When $\alpha=0$ or $\beta=\epsilon_2$, Proposition \ref{prop} gives the formula for $g_{n+k,n}=[x^n]\hat{p}_{n+k}(x)$. The weight $u_m(\mu)$ given by Definition \ref{deff} can be considered as a modification of the weight $w(\mu)$ given by Definition \ref{defw}. They manifestly coincide when $\alpha=0$ or $\beta=\epsilon_2$ (Remark \ref{mod}).

\begin{Rem}\label{mod}
Let $k=s_m(\mu)$ and assume $\alpha=0$. If $k=0$, then trivially $u_m(\mu)=w(\mu)$. If $k>0$, then both $u_m(\mu)$ and $w(\mu)$ have a factor $Z_0=X_0=\alpha$ so they are both zero. Now assume $\beta=\epsilon_2$ then $Y_i$ equals $[i+1]_q \beta$. This gives 

\begin{align*}
\frac{w(\mu)}{u_m(\mu)}=&\frac{(\prod\limits_{i=1}^{l}Z_{\mu_{l+1-i}+2(i-1)})}{(\prod\limits_{i=1}^{l-k}Z_{\mu_{l+1-i}+2(i-1)})(\binom{m+l}{k}_q Y_0\cdots Y_{k-1})}=\frac{(\prod\limits_{i=l-k+1}^{l}Z_{\mu_{l+1-i}+2(i-1)})}{(\binom{m+l}{k}_q Y_0\cdots Y_{k-1})} \\ =&\frac{Y_{m+l-k}\cdots Y_{m+l-1}}{(\binom{m+l}{k}_q Y_0\cdots Y_{k-1})}=\frac{\beta^k [m+l-k+1]_q \cdots [m+l]_q}{\beta^k \binom{m+l}{k}_q [1]_q \cdots [k]_q}=1.
\end{align*}
We conclude that $w(\mu)=u_m(\mu)$ when $\alpha=0$ or $\beta=\epsilon_2$. 
\end{Rem}

\subsection{Proof of Theorem \ref{main1}}

We will first prove Theorem \ref{main1} for $g_{k,0}$.  It suffices to prove that (\ref{0}) satisfies the recurrence $g_{k+1,0}=b_k g_{k,0}-\lambda_{k} g_{k-1,0}$ (Proposition \ref{prop1}). The base cases $g_{0,0}=1$ and $g_{1,0}=X_0+Y_0=\alpha+\beta$ are trivially satisfied. We prepare with a lemma. 
\begin{Lemma} \label{lem}
The following equality holds ($k \geq 1$)
$$\binom{k+1}{i}_q X_{i-1} Y_{k-i}=\binom{k}{i-1}_q X_{k} Y_{k-i}+\binom{k}{i}_q X_{i-1} Y_{k}-\lambda_k \binom{k-1}{i-1}_q.  $$
\end{Lemma}
\begin{proof}
Moving the middle term on the right hand side to the left, the left hand side becomes
\begin{align*}
\binom{k+1}{i}_q X_{i-1}Y_{k-i}-\binom{k}{i}_q X_{i-1} Y_{k}&=X_{i-1}(\binom{k+1}{i}_q Y_{k-i}-\binom{k}{i}_q Y_{k}) \\ &=X_{i-1} (q^{k-i}\binom{k}{i-1}_q(\beta-\epsilon_2)).
\end{align*}
The remaining right hand side becomes
\begin{align*}
&\binom{k}{i-1}_q X_{k}Y_{k-i}-(X_k Y_{k-1}-\alpha (\beta-\epsilon_2) q^{k-1})\binom{k-1}{i-1}_q \\=&X_k(\binom{k}{i-1}_q Y_{k-i}-\binom{k-1}{i-1}_q Y_{k-1})+\alpha(\beta-\epsilon_2)q^{k-1}\binom{k-1}{i-1}_q \\ =&X_k(q^{k-i}\binom{k-1}{i-2}_q(\beta-\epsilon_2))+\alpha(\beta-\epsilon_2)q^{k-1}\binom{k-1}{i-1}_q \\ =&q^{k-i}(\beta-\epsilon_2)(\binom{k-1}{i-2}_q X_k+\alpha q^{i-1} \binom{k-1}{i-1}_q) \\ =&q^{k-i}(\beta-\epsilon_2)(\alpha (q^k \binom{k-1}{i-2}_q+ q^{i-1} \binom{k-1}{i-1}_q))+ \epsilon_1 [k]_q \binom{k-1}{i-2}_q) \\ =&q^{k-i}(\beta-\epsilon_2)(\alpha (q^{i-1}\binom{k}{i-1}_q)+ \epsilon_1 [i-1]_q \binom{k}{i-1}_q)\\=& q^{k-i}(\beta-\epsilon_2)(\binom{k}{i-1}_q X_{i-1})).
\end{align*} 
So the equality holds.
\end{proof}

\begin{Prop} \label{prop1}
The following equality holds ($k \geq 1$)
\begin{align*}
 \sum_{i=0}^{k+1} \binom{k+1}{i}_q  X_0\cdots X_{i-1} Y_0\cdots Y_{k-i}=&b_k(\sum_{i=0}^{k} \binom{k}{i}_q  X_0\cdots X_{i-1} Y_0\cdots Y_{k-i-1})\\&-\lambda_k(\sum_{i=0}^{k-1} \binom{k-1}{i}_q  X_0\cdots X_{i-1} Y_0\cdots Y_{k-i-2}).
\end{align*}
\end{Prop}
\begin{proof}
Multiplying the equality in Lemma \ref{lem} with $(X_0\cdots X_{i-2}Y_0\cdots Y_{k-i-1})$ gives
\begin{align}\label{444}
   &\binom{k+1}{i}_q  X_0\cdots X_{i-1} Y_0\cdots Y_{k-i}=(\binom{k}{i-1}_q  X_0\cdots X_{i-2} Y_0\cdots Y_{k-i-1})(X_k)\\&+(\binom{k}{i}_q  X_0\cdots X_{i-1} Y_0\cdots Y_{k-i-2})(Y_k)-(\binom{k-1}{i-1}_q  X_0\cdots X_{i-2} Y_0\cdots Y_{k-i-1})(\lambda_k). \nonumber    
\end{align}
 Summing \eqref{444} for $i$ from 0 to $(k+1)$ gives a desired equality.
\end{proof}

To prove Theorem \ref{main1}, it remains to show that for $n>0$, the formula satisfies the recurrence $g_{n+k+1,n}=g_{n+k,n-1}+b_{n+k} g_{n+k,n}-\lambda_{n+k} g_{n+k-1,n}$. In other words, we will show the identity
\begin{align} \label{equ}
\sum_{\mu \subseteq (k+1) \times (2n+1)} u_n (\mu)=&\sum_{\mu \subseteq (k+1) \times (2n-1)} u_{n-1} (\mu)+b_{n+k}(\sum_{\mu \subseteq (k) \times (2n+1)} u_n (\mu)) \\
&-\lambda_{n+k} (\sum_{\mu \subseteq (k-1) \times (2n+1)} u_n (\mu)).\nonumber 
\end{align} 
The middle term on the right hand side of \eqref{equ} becomes
\begin{align*}
  &b_{n+k}(\sum_{\mu \subseteq (k) \times (2n+1)} u_n (\mu))=X_{n+k}(\sum_{\mu \subseteq (k) \times (2n+1)} u_n (\mu))+Y_{n+k}(\sum_{\mu \subseteq (k) \times (2n+1)} u_n (\mu))  \\=&X_{n+k}(\sum_{\mu \subseteq (k) \times (2n)} u_n (\mu)+\sum_{\substack{\mu \subseteq (k) \times (2n+1) \\ \mu_1=2n+1}} u_n (\mu))+Y_{n+k}(\sum_{\mu \subseteq (k) \times (2n+1)} u_n (\mu)) \\
  =&X_{n+k}(\sum_{\mu \subseteq (k) \times (2n)} u_n (\mu))+Y_{n+k}(\sum_{\mu \subseteq (k) \times (2n+1)} u_n (\mu))\\&+X_{n+k}(\sum_{\mu' \subseteq (k-1) \times (2n+1)} u_n ((2n+1,\mu'))).
\end{align*}
Plugging this to \eqref{equ} gives
\begin{align} \label{equ2}
&\sum_{\mu \subseteq (k+1) \times (2n+1)} u_n (\mu)=\sum_{\mu \subseteq (k+1) \times (2n-1)} u_{n-1} (\mu)+X_{n+k}(\sum_{\mu \subseteq (k) \times (2n)} u_n (\mu)) \\
&+Y_{n+k}(\sum_{\mu \subseteq (k) \times (2n+1)} u_n (\mu))+(\sum_{\mu \subseteq (k-1) \times (2n+1)} (X_{n+k} u_n ((2n+1,\mu))-\lambda_{n+k} u_n (\mu))).\nonumber 
\end{align} 
For $\mu \subseteq (k+1) \times (2n+1)$ we define
\begin{equation*}
\bar{u}_n(\mu) = \begin{cases}
u_{n-1}(\mu) &\text{if $\mu \subseteq (k+1) \times (2n-1)$}\\
X_{n+k} u_n (\mu') &\text{if $\mu=(2n,\mu')$ for some $\mu' \subseteq (k) \times (2n)$}\\
Y_{n+k} u_n (\mu') &\text{if $\mu=(2n+1,\mu')$ for some $\mu' \subseteq (k) \times (2n+1)$},
\end{cases}
\end{equation*}
and for $\mu \subseteq (k-1) \times (2n+1)$ we define
$$v_i (\mu)=X_{n+k} u_n ((2n+1,\mu))-\lambda_{n+k} u_n (\mu).$$
Then \eqref{equ2} is represented as follows 
\begin{align} \label{equ3}
\sum_{\mu \subseteq (k+1) \times (2n+1)} u_n (\mu)=\sum_{\mu \subseteq (k+1) \times (2n+1)} \bar{u}_n (\mu)+\sum_{\mu \subseteq (k-1) \times (2n+1)}v_i (\mu). 
\end{align} 
To prove \eqref{equ3} we first partition the sets $\{\mu \subseteq (k+1) \times (2n+1)\}$ and $\{\mu \subseteq (k-1) \times (2n+1)\}$.

\begin{Def}
Define $\hat{B}_{n+k,n}$ to be a subset of $\{\mu \subseteq (k+1) \times (2n+1)\}$ consisting of $\mu$ such that $s_n(\mu)=0$ and $s_{n-1}(\mu)=0$. For a partition $\mu \subseteq (k+1) \times (2n-1)$ such that $s_{n-1}(\mu)=l>0$, denoting $\mu=((2n-1)^l,\mu')$, we define $B^{\mu}_{n+k,n}$ to be
$$B^{\mu}_{n+k,n}=\{((2n-1)^l,\mu'),((2n+1),(2n-1)^{l-1},\mu'),\cdots,((2n+1)^l,\mu')\}.$$
\end{Def}

\begin{Def}
For a partition $\mu \subseteq (k-1)\times(2n+1)$ such that $\mu_{k-1} \geq 2$, define $\bar{s}_n(\mu)$ to be a mininum of two numbers $|\{i | \mu_i=2\}|$ and $|\{i|\mu_i=2n+1\}|$. If $\bar{s}_n(\mu)=l$, denoting $\mu=(\mu',2^l)$, we define a set $C^{\mu}_{n+k,n}$ to be 
$$C^{\mu}_{n+k,n}=\{(\mu',2^l),(\mu',2^{l-1},0),\cdots,(\mu',0^l)\}.$$
If $l=0$, then $C^{\mu}_{n+k,n}$ consists of a single element $\mu$.
\end{Def}

\begin{Ex}
Consider a partition $\mu=(3,3,1,0,0,0)$. Since $s_1(\mu)=2$, we have
$$B^{(3,3,1,0,0,0)}_{7,2}=\{(3,3,1,0,0,0),(5,3,1,0,0,0),(5,5,1,0,0,0)\}.$$
For a partition $\mu=(5,5,2,2)$, we have
$$C^{(5,5,2,2)}_{7,2}=\{(5,5,2,2),(5,5,2,0),(5,5,0,0)\}.$$
\end{Ex}

\begin{Def}
We define $B^{X}_{n+k,n}$, $B^{Y}_{n+k,n}$, $C^{X}_{n+k,n}$ and $C^{Y}_{n+k,n}$ as follows
\begin{align*}
     &B^{X}_{n+k,n}=\{ \mu \subseteq (k+1)\times(2n+1) |\text{ $s_n(\mu)=l>0$ and $\mu=((2n+1)^l,2n,\mu')$ for some $\mu'$} \} \\
      &B^{Y}_{n+k,n}=\{ \mu \subseteq (k+1)\times(2n+1) |\text{ $s_n(\mu)=l>0$ and $\mu=((2n+1)^{l+1},\mu')$ for some $\mu'$} \}\\
       &C^{X}_{n+k,n}=\{ \mu \subseteq (k-1)\times(2n+1) |\text{ $s_n(\mu)=l$ and $\mu=(\mu',0^{l+1})$ for some $\mu'$} \} \\
     &C^{Y}_{n+k,n}=\{ \mu \subseteq (k-1)\times(2n+1) |\text{ $s_n(\mu)=l$ and $\mu=(\mu',1,0^l)$ for some $\mu'$} \} .  
\end{align*}

\end{Def}
\begin{Prop}\label{decom}
The set  $\{\mu \subseteq (k+1) \times (2n+1)\}$ is a disjoint union of $B$'s and the set $\{\mu \subseteq (k-1) \times (2n+1)\}$ is a disjoint union of $C$'s. That is
\begin{align*}
    &\{\mu \subseteq (k+1) \times (2n+1)\}=\hat{B}_{n+k,n}\cupdot(
\bigcupdot_{\substack{\nu \subseteq (k+1)\times(2n-1) \\ s_{n-1}(\nu)>0}}B^{\nu}_{n+k,n} )\cupdot B^{X}_{n+k,n}\cupdot B^{Y}_{n+k,n}\\
&\{\mu \subseteq (k-1) \times (2n+1)\}=(\bigcupdot_{\substack{\nu \subseteq (k-1)\times(2n+1) \\ \nu_{k-1}\geq2}}C^{\nu}_{n+k,n} )\cupdot C^{X}_{n+k,n}\cupdot C^{Y}_{n+k,n}.
\end{align*}
\end{Prop}
\begin{proof}
For $\mu \subseteq (k+1) \times (2n+1)$, if $s_n(\mu)=0$ and $s_{n-1}(\mu)=0$ then $\mu \in \hat{B}_{n+k,n}$. If $s_{n-1}(\mu)>0$,  then $\mu \in B^{\mu}_{n+k,n}$. Now assume $s_{n}(\mu)=l>0$. By the definition of $B^{X}_{n+k,n}$ and $B^{Y}_{n+k,n}$, we have $\mu \in B^{X}_{n+k,n}$ if $\mu_{l+1}=2n$ and  $\mu \in B^{Y}_{n+k,n}$ if $\mu_{l+1}=2n+1$. For the remaining case $\mu_{l+1} \leq 2n-1$, let $\nu=((2n-1)^l,\mu_{l+1},\cdots,\mu_{k+1})$. Then $\nu$ is a partition inside $(k+1)\times(2n-1)$ with $s_{n-1}(\nu)\geq l>0$. Thus we have $\mu \in B^{\nu}_{n+k,n}$. This proves the first statement. 

For the second statement, take $\mu \subseteq (k-1)\times (2n+1)$ such that $s_n(\mu)=l$. If $\mu_{k-1-l}=0$ then $\mu \in C^{X}_{n+k,n}$ and if $\mu_{k-1-l}=1$ then $\mu \in C^{Y}_{n+k,n}$. For the remaining case  $\mu_{k-1-l} \geq 2$, let $\nu=(\mu_1,\cdots,\mu_{k-1-l},2^l)$, then we have $\mu \in C^{\nu}_{n+k,n}$.  
\end{proof}

We have decompositions of the sets  $\{\mu \subseteq (k+1) \times (2n+1)\}$ and  $\{\mu \subseteq (k-1) \times (2n+1)\}$. The next proposition relates these two decompositions.

\begin{Prop}\label{bi}
The followings hold.

(a) There exists a bijection between $\{\mu \subseteq (k+1)\times(2n-1) |  s_{n-1}(\mu)>0\}$ and \\ $\{\mu \subseteq (k-1)\times(2n+1) |  \mu_{k-1}\geq2\}$.

(b) There exists a bijection between $B^{X}_{n+k,n}$ and $C^{X}_{n+k,n}$. 

(c) There exists a bijection between $B^{Y}_{n+k,n}$ and $C^{Y}_{n+k,n}$.
\end{Prop}
\begin{proof}
(a) Given an element $\mu$ in the first set, the partition $(\mu_2+2,\cdots,\mu_{k}+2)$ is in the second set. Conversely given an element $\nu$ in the second set, the partition $(2n-1,\nu_1-2,\cdots,\nu_{k-1}-2,0)$ is in the first set. This gives a bijection.

(b) Given an element $\mu$ in the first set, let $s_n(\mu)=l>0$ and denote $\mu=((2n+1)^l,2n,\mu')$. Note that the tail of $\mu'$ contains at least $l$ 0's. So we write $\mu=((2n+1)^l,2n,\mu'',0^l)$. Then the partition $\nu=((2n+1)^{l-1},\mu'',0^{l})$ belongs to the second set. Conversely, given an element $\nu$ in the second set, let $s_n(\nu)=l'$ (possibly zero). Likewise, we can write $\nu=((2n+1)^{l'},\nu'',0^{l'+1})$. We send this to the partition $\mu=((2n+1)^{l'+1},2n,\nu'',0^{l'+1})$ then $s_n(\mu)=l'+1>0$ which implies that $\mu$ belongs to the first set. These processes are inverse to each other thus give a bijection.

(c) The argument goes similarly with that of (b). The partition $\mu=((2n+1)^{l+1},\mu'',0^l)$ goes to the partition $\nu=((2n+1)^{l-1},\mu'',1,0^{l-1})$ and conversely, the partition $\nu=((2n+1)^{l'},\nu'',1,0^{l'})$ goes to the partition $\mu=((2n+1)^{l'+2},\nu'',0^{l'+1})$.
\end{proof}

\begin{Prop}\label{re}
The following identities hold.

(a) For a partition $\mu\in\hat{B}_{n+k,n}$, we have $u_n(\mu)=\bar{u}_n(\mu)$.

(b) For partitions $\mu \subseteq (k+1)\times(2n-1)$ such that $s_{n-1}(\mu)>0$ and $\nu \subseteq (k-1)\times(2n+1)$ such that  $\nu_{k-1}\geq2$ which are corresponding pair under the bijection in Proposition \ref{bi} (a), we have 
$$\sum_{\mu'\in B^{\mu}_{n+k,n}}u_n(\mu')=\sum_{\mu'\in B^{\mu}_{n+k,n}}\bar{u}_n(\mu')+\sum_{\nu'\in C^{\nu}_{n+k,n}}v_n(\nu').$$

(c) For partitions $\mu\in B^{X}_{n+k,n}$ and $\nu\in C^{X}_{n+k,n}$ which are corresponding pair under the bijection in Proposition \ref{bi} (b), we have 
$$u_n(\mu)=\bar{u}_n(\mu)+v_n(\nu).$$

(d) For partitions $\mu\in B^{Y}_{n+k,n}$ and $\nu\in C^{Y}_{n+k,n}$ which are corresponding pair under the bijection in Proposition \ref{bi} (c), we have 
$$u_n(\mu)=\bar{u}_n(\mu)+v_n(\nu).$$

\end{Prop}
Note that (\ref{equ3}) follows from Proposition \ref{decom} and Proposition \ref{re}. So it suffices to prove Proposition \ref{re}. To do that, we first compute $v_n(\nu)$ explicitly. 
\begin{Lemma}\label{as} The following holds.

(a) For a partition $\nu \in C^{X}_{n+k,n}$ with $s_n(\nu)=l$, we have 
$$v_n(\nu)=X_0\cdots X_l X_{n+k-l-1} (\prod\limits_{i=l+1}^{k-2-l} Z_{\nu_i+2(k-1-i)})Y_0 \cdots Y_{l-1}q^l \binom{n+k}{l+1}_q (\beta-\epsilon_2).$$
(b) For a partition $\nu \subseteq (k-1)\times(2n+1)$ such that $\nu \notin C^{X}_{n+k,n}$ with $s_n(\nu)=l$, we have 
$$v_n(\nu)=X_0\cdots X_{l} (\prod\limits_{i=l+1}^{k-1-l} Z_{\nu_i+2(k-1-i)}) Y_0 \cdots Y_{l-1} q^{n+k-l-1} \binom{n+k}{l}_q (\beta-\epsilon_2).$$

\end{Lemma}
\begin{proof}
(a) Denoting $\nu=((2n+1)^l,\nu_{l+1},\cdots,\nu_{k-2-l},0^{l+1})$, we have $$u_n(\nu)=X_0\cdots X_l (\prod\limits_{i=l+1}^{k-2-l} Z_{\nu_i+2(k-1-i)})\binom{n+k-1}{l}_q Y_0 \cdots Y_{l-1}.$$ And since $s_n(((2n+1),\nu))=l+1$, we have $$u_n(((2n+1),\nu))=X_0\cdots X_l (\prod\limits_{i=l+1}^{k-2-l} Z_{\nu_i+2(k-1-i)})\binom{n+k}{l+1}_q Y_0 \cdots Y_{l}.$$ Setting the common factor $M=X_0\cdots X_l (\prod\limits_{i=l+1}^{k-2-l} Z_{\nu_i+2(k-1-i)}) Y_0 \cdots Y_{l-1}$, we have
\begin{align*}
    &v_i(\nu)\\=&X_{n+k} u_n(((2n+1),\nu))-\lambda_{n+k} u_n(\nu)=X_{n+k}(M \binom{n+k}{l+1}_q Y_l)-\lambda_{n+k}(M\binom{n+k-1}{l}_q) \\=& M(\binom{n+k}{l+1}_q X_{n+k} Y_l-(Y_{n+k-1}X_{n+k}-\alpha(\beta-\epsilon_2)q^{n+k-1})\binom{n+k-1}{l}_q)
    \\=&M(X_{n+k}(\binom{n+k}{l+1}_q Y_l-\binom{n+k-1}{l}_q Y_{n+k-1})+\alpha(\beta-\epsilon_2)q^{n+k-1}\binom{n+k-1}{l}_q)
    \\=&M(X_{n+k}(q^{l}\binom{n+k-1}{l+1}_q (\beta-\epsilon_2))+\alpha(\beta-\epsilon_2)q^{n+k-1}\binom{n+k-1}{l}_q)
    \\=&M(\beta-\epsilon_2)q^l(X_{n+k}\binom{n+k-1}{l+1}_q+\alpha q^{n+k-l-1}\binom{n+k-1}{l}_q)) 
    \\=&M(\beta-\epsilon_2)q^l(\alpha(q^{n+k} \binom{n+k-1}{l+1}+q^{n+k-l-1}\binom{n+k-1}{l}_q)+\epsilon_1 \binom{n+k-1}{l+1}[n]_q) 
     \\=&M(\beta-\epsilon_2)q^l(\binom{n+k}{l+1}_q X_{n+k-l-1}) 
     \\=&X_0\cdots X_l X_{n+k-l-1} (\prod\limits_{i=l+1}^{k-2-l} Z_{\nu_i+2(k-1-i)})Y_0 \cdots Y_{l-1}q^l \binom{n+k}{l+1}_q (\beta-\epsilon_2).
\end{align*}
(b) Denoting $\nu=((2n+1)^l,\nu_{l+1},\cdots,\nu_{k-1-l},0^{l})$, we have $$u_n(\nu)=X_0\cdots X_{l-1} (\prod\limits_{i=l+1}^{k-1-l} Z_{\nu_i+2(k-1-i)})\binom{n+k-1}{l}_q Y_0 \cdots Y_{l-1}.$$ And since $s_n(((2n+1),\nu))=l$, we have 
\begin{align*}
  u_n(((2n+1),\nu))=&X_0\cdots X_{l-1} (\prod\limits_{i=l}^{k-1-l} Z_{\nu_i+2(k-1-i)})\binom{n+k}{l}_q Y_0 \cdots Y_{l-1} \\
  =&X_0\cdots X_{l-1} (\prod\limits_{i=l+1}^{k-1-l} Z_{\nu_i+2(k-1-i)})\binom{n+k}{l}_q Y_0 \cdots Y_{l-1}Z_{\nu_{l}+2(k-1-l)}
   \\
  =&X_0\cdots X_{l-1} (\prod\limits_{i=l+1}^{k-1-l} Z_{\nu_i+2(k-1-i)})\binom{n+k}{l}_q Y_0 \cdots Y_{l-1}Y_{n+k-l-1}.  
\end{align*}
Setting the common factor $M=X_0\cdots X_{l-1} (\prod\limits_{i=l+1}^{k-1-l} Z_{\nu_i+2(k-1-i)}) Y_0 \cdots Y_{l-1}$, we have
\begin{align*}
    v_n(\nu)=&M( \binom{n+k}{l}_q X_{n+k} Y_{n+k-l-1}-(Y_{n+k-1}X_{n+k}-\alpha(\beta-\epsilon_2)q^{n+k-1})\binom{n+k-1}{l}_q) \\
    =&M(\beta-\epsilon_2)q^{n+k-l-1}(\binom{n+k}{l}_q X_l) \\
    =&X_0\cdots X_{l} (\prod\limits_{i=l+1}^{k-1-l} Z_{\nu_i+2(k-1-i)}) Y_0 \cdots Y_{l-1} q^{n+k-l-1} \binom{n+k}{l}_q (\beta-\epsilon_2).
\end{align*}

\end{proof}

The following lemma will be used for the proof of Proposition \ref{re} (b). 

\begin{Lemma}\label{219}
For $l$ and $n$ such that $l \geq 0$ and $n-l-1 \geq 0$, we have
\begin{align*}
    &Y_{n-l-1}\cdots Y_{n-1}+\sum_{i=1}^{l+1}Y_{n-l-1}\cdots Y_{n-i-1} \binom{n+1}{i}_q Y_0 \cdots Y_{i-1}
    \\=&\binom{n}{l+1}_qY_0 \cdots Y_l+Y_n(\sum_{i=1}^{l+1}Y_{n-l-1}\cdots Y_{n-i-1} \binom{n}{i-1}_q Y_0 \cdots Y_{i-2})\\&+(\beta-\epsilon_2)(\sum_{i=0}^{l}q^{n-i-1}Y_{n-l-1}\cdots Y_{n-i-2} \binom{n}{i}_q Y_0 \cdots Y_{i-1}).
\end{align*}
\end{Lemma}
\begin{proof}
For $l=0$, the identity becomes $Y_{n-1}+\binom{n+1}{1}_q Y_0=\binom{n}{1}_q Y_0+Y_n+(\beta-\epsilon_2)q^{n-1}$, which can be checked by a direct computation. Now assume that the identity holds for $(l-1)$ and for all valid $n$. From the identity corresponding to $(l-1)$ and $n$, we multiply both sides with $Y_{n-l-1}$ which gives 
\begin{align}\label{27}
    &Y_{n-l-1}\cdots Y_{n-1}+\sum_{i=1}^{l}Y_{n-l-1}\cdots Y_{n-i-1} \binom{n+1}{i}_q Y_0 \cdots Y_{i-1}
    \\=&\binom{n}{l}_q Y_0 \cdots Y_{l-1}Y_{n-l-1}+Y_n(\sum_{i=1}^{l}Y_{n-l-1}\cdots Y_{n-i-1} \binom{n}{i-1}_q Y_0 \cdots Y_{i-2})\nonumber\\&+(\beta-\epsilon_2)(\sum_{i=0}^{l-1}q^{n-i-1}Y_{n-l-1}\cdots Y_{n-i-2} \binom{n}{i}_q Y_0 \cdots Y_{i-1}).\nonumber
\end{align}
We also have the identity 
$$\binom{n+1}{l+1}_q Y_l=\binom{n}{l+1}_q Y_l-\binom{n}{l}_q Y_{n-l-1}+\binom{n}{l}_q Y_n+q^{n-l-1}\binom{n}{l}_q(\beta-\epsilon_2), $$
that can be checked by a direct computation. Multiplying both sides with $Y_0\cdots Y_{l-1}$ gives
\begin{align}\label{28}
    \binom{n+1}{l+1}_q Y_0\cdots Y_l=&\binom{n}{l+1}_q Y_0\cdots Y_l-\binom{n}{l}_q Y_0\cdots Y_{l-1}Y_{n-l-1}\\&+\binom{n}{l}_q Y_0\cdots Y_{l-1}Y_n+q^{n-l-1}\binom{n}{l}_qY_0\cdots Y_{l-1}(\beta-\epsilon_2).\nonumber
\end{align}
Adding (\ref{27}) and (\ref{28}) gives the identity corresponding to $l$ and $n$. The proof follows from the induction.
\end{proof}

\textit{Proof of Proposition \ref{re}.}

(a) It is trivial to verify.

(b) Denote $\bar{s}_n(\nu)=l$ and $\nu=((2n+1)^l,\nu_{l+1},\cdots,\nu_{k-1-l},2^l)$. Then the corresponding partition $\mu$ is 
$((2n-1)^{l+1},(\nu_{l+1}-2),\cdots,(\nu_{k-1-l}-2),0^{l+1})$. The elements of $C^{\nu}_{n+k,n}$ are denoted as
$\nu^i=((2n+1)^l,\nu_{l+1},\cdots,\nu_{k-1-l},2^{l-i},0^i)$ where $i$ ranges from 0 to $l$. As $s_n(\nu^{i})=i$, by Lemma \ref{as} we have

\begin{align*}
    &v_n(\nu^{i})=X_0\cdots X_{i} (\prod\limits_{j=i+1}^{k-1-i} Z_{(\nu^{i})_j+2(k-1-j)}) Y_0 \cdots Y_{i-1} q^{n+k-i-1} \binom{n+k}{i}_q (\beta-\epsilon_2)
    \\=&X_0\cdots X_{i} (\prod\limits_{j=i+1}^{k-1-l} Z_{(\nu^{i})_j+2(k-1-j)})(\prod\limits_{j=k-l}^{k-1-i} Z_{(\nu^{i})_j+2(k-1-j)}) Y_0 \cdots Y_{i-1} q^{n+k-i-1} \binom{n+k}{i}_q (\beta-\epsilon_2) \\=&X_0\cdots X_{l} (\prod\limits_{j=i+1}^{k-1-l} Z_{(\nu^{i})_j+2(k-1-j)}) Y_0 \cdots Y_{i-1} q^{n+k-i-1} \binom{n+k}{i}_q (\beta-\epsilon_2)
    \\=&X_0\cdots X_{l} (\prod\limits_{j=i+1}^{l} Z_{(\nu^{i})_j+2(k-1-j)})(\prod\limits_{j=l+1}^{k-1-l} Z_{(\nu^{i})_j+2(k-1-j)}) Y_0 \cdots Y_{i-1} q^{n+k-i-1} \binom{n+k}{i}_q (\beta-\epsilon_2)
     \\=&X_0\cdots X_{l} (\prod\limits_{j=l+1}^{k-1-l} Z_{\nu_{j+2(k-1-j)}})(Y_{n+k-l-1}\cdots Y_{n+k-i-2}) Y_0 \cdots Y_{i-1} q^{n+k-i-1} \binom{n+k}{i}_q (\beta-\epsilon_2).
\end{align*}

And the elements of $B^{\mu}_{n+k,n}$ are denoted as  $\mu^i=((2n+1)^{i},(2n-1)^{l+1-i},(\nu_{l+1}-2),\cdots,(\nu_{k-1-l}-2),0^{l+1})$ where $i$ ranges from 0 to $(l+1)$. We have
\begin{align*}
    &u_n(\mu^0)=X_0\cdots X_{l} (\prod\limits_{j=l+1}^{k-1-l} Z_{\nu_{j+2(k-1-j)}})(Y_{n+k-l-1}\cdots Y_{n+k-1})\\
    &\bar{u}_n(\mu^0)=X_0\cdots X_{l} (\prod\limits_{j=l=1}^{k-1-l} Z_{\nu_{j+2(k-1-j)}})(\binom{n+k}{l+1}_q Y_0\cdots Y_l),
\end{align*}
and for $i$ from 1 to $(l+1)$ we have
\begin{align*}
    &u_n(\mu^i)=X_0\cdots X_{l} (\prod\limits_{j=l+1}^{k-1-l} Z_{\nu_{j+2(k-1-j)}})(Y_{n+k-l-1}\cdots Y_{n+k-i-1} \binom{n+k+1}{i}_q Y_0 \cdots Y_{i-1})\\
    &\bar{u}_n(\mu^i)=X_0\cdots X_{l} (\prod\limits_{j=l+1}^{k-1-l} Z_{\nu_{j+2(k-1-j)}})(Y_{n+k-l-1}\cdots Y_{n+k-i-1} \binom{n+k}{i-1}_q Y_0 \cdots Y_{i-2})Y_n.
\end{align*}
Taking out the common factor $(X_0\cdots X_{l} (\prod\limits_{j=l+1}^{k-1-l} Z_{\nu_{j+2(k-1-j)}}))$, the desired identity follows from Lemma \ref{219}.

(c) Denote $s_n(\nu)=l$ and $\nu=((2n+1)^l,\nu_{l+1},\cdots,\nu_{k-2-l},0^{l+1})$. Then the corresponding partition $\mu$ is $((2n+1)^{l+1},2n,\nu_{l+1},\cdots,\nu_{k-2-l},0^{l+1})$. We have 
\begin{align*}
    u_n(\mu)=&X_0\cdots X_l(\prod\limits_{j=l+2}^{k-l} Z_{\mu_j+2(k+1-j)})\binom{n+k+1}{l+1}_q Y_0\cdots Y_l
    \\=&X_0\cdots X_l X_{n+k-l-1}(\prod\limits_{j=l+1}^{k-2-l} Z_{\nu_j+2(k-1-j)})\binom{n+k+1}{l+1}_q Y_0\cdots Y_l,
\end{align*}
and 
\begin{align*}
    \bar{u}_n(\mu)=X_0\cdots X_l X_{n+k-l-1}(\prod\limits_{j=l+1}^{k-2-l} Z_{\nu_j+2(k-1-j)})\binom{n+k}{l}_q Y_0\cdots Y_{l-1}Y_{n+k}.
\end{align*}
So we have 
\begin{align*}
     &u_n(\mu)-\bar{u}_n(\mu)\\=&(X_0\cdots X_l X_{n+k-l-1}(\prod\limits_{j=l+1}^{k-2-l} Z_{\nu_j+2(k-1-j)}) Y_0\cdots Y_{l-1})(\binom{n+k+1}{l+1}_q Y_l-\binom{n+k}{l}_q Y_{n+k}) \\=&(X_0\cdots X_l X_{n+k-l-1}(\prod\limits_{j=l+1}^{k-2-l} Z_{\nu_j+2(k-1-j)}) Y_0\cdots Y_{l-1})q^l\binom{n+k}{l+1}_q(\beta-\epsilon_2).
\end{align*}
The proof follows from Lemma \ref{as}.

(d) Denote $s_n(\nu)=l$ and $\nu=((2n+1)^l,\nu_{l+1},\cdots,\nu_{k-2-l},1,0^{l})$. Then the corresponding partition $\mu$ is $((2n+1)^{l+2},2n,\nu_{l+1},\cdots,\nu_{k-2-l},0^{l+1})$. The proof follows similarly with (c).

\section{Proof of Theorem \ref{main2}} Our goal in this section is to prove Theorem \ref{main2}, which gives a combinatorial formula for the coefficients of the transformed Al-Salam-Chihara polynomials as a weighted sum over pairs of sets. In Section \ref{bijection}, we start by providing a bijective proof of a combinatorial identity \eqref{simple}, which motivates the formula in Theorem \ref{main2}. In Section \ref{gqb}, we study generalized $q$-binomial coefficients $M_n^{\mu}(b)$ (Definition \ref{weight}) to establish key lemmas (Lemma \ref{ana}, Lemma \ref{t}) for the proof of Theorem \ref{main2}. In Section \ref{pt}, we finish the proof.
\subsection{Bijective proof of the identity \texorpdfstring{$\eqref{simple}$}{Lg}}\label{bijection}
\begin{Def}
For non-negative integers $n$, $a$ and $b$, we define $T(n,a,b)$ to be a set of pairs $(S_1,S_2)$ such that $S_1\subseteq \{0,\cdots,(n+a-1)\}$ with $|S_1|=a$ and $S_2\subseteq \{0,\cdots,(n+a+b-1)\}$ with $|S_2|=b$. For a set $S$ with integer elements, we define $\sum S=\sum\limits_{i \in S}i$.  
\end{Def}
 The identity \eqref{simple} can be rephrased as follows
\begin{align}\label{bb}
\sum\limits_{(A_1,B_1)\in T(n,a,b)}q^{\sum A_1+\sum B_1}=\sum\limits_{(B_2,A_2)\in T(n,b,a)}q^{\sum A_2+\sum B_2}
\end{align}
The identity \eqref{bb} shows that there exists a bijection between $T(n,a,b)$ and $T(n,b,a)$ that preserves the sum. Now we construct a such bijection. Recall from Definition \ref{S} that $S(k)$ is the $k$-th smallest element of the set $(\{0\}\cup\mathbb{N})-S$ and $\lambda_S$ is $(i_1,i_2-1,\cdots,i_s-s+1)$ where $S=\{i_1<\cdots<i_s\}$.
\begin{Def}\label{bijection1}
For $(A_1,B_1) \in T(n,a,b)$, we denote $A_1=\{ i_1<\cdots<i_a\}$, $B_1 \cap \{n+b,\cdots,n+b+a-1\}=\{n+b+a-j_k<\cdots<n+b+a-j_1\}$ and $\mu=\lambda_{A_1}$. We define a map $\psi_{n,a,b}$ from $T(n,a,b)$ as follows 
\begin{align*}
    &(B_2,A_2)=\psi_{n,a,b}((A_1,B_1))\\
    &A_2=\{i_m |m \neq j_l\} \cup \{n+a+b-(B_1(\mu_{j_l}+l)-(i_{j_l}-j_l))|l=1,\cdots,k\}\\
    &B_2=(B_1\cap\{0,\cdots,n+b-1\})\cup\{B_1(\mu_{j_l}+l)|l=1,\cdots,k\}.
\end{align*}
Note that we have $\sum A_1+\sum B_1=\sum A_2+\sum B_2$ from the construction.
\end{Def}
\begin{Prop}
With the notation in Definition \ref{bijection1}, we have $(B_2,A_2) \in T(n,b,a)$. And we have $(A_1,B_1)=\psi_{n,b,a}((B_2,A_2))$.
\end{Prop}
\begin{proof} We have $B_1(\mu_{j_1}+1)<\cdots<B_1(\mu_{j_k}+k)$ and since $\mu_{j_k}\leq n$, we have $B_1(\mu_{j_k}+k)\leq B_1(n+k)$. As $|B_1\cap\{0,\cdots,(n+b-1)\}|=b-k$, we have $B_1(n+k)\leq (n+b-1)$, which implies  $B_2\subseteq \{0,\cdots,(n+b-1)\}$ with $|B_2|=b$. Now we set $r_l=B_1(\mu_{j_l}+l)-(i_{j_l}-j_l)$.
Since $B_1(\mu_{j_l}+l)$ is the $(\mu_{j_l}+l)$-th smallest element in a set $\{0,\cdots,n+b-1\}-B_1$, there are $B_1(\mu_{j_l}+l)-(\mu_{j_l}+l+1)$ elements in $B_1$ smaller than $B_1(\mu_{j_l}+l)$. Also $B_1(\mu_{j_1}+1),\cdots,B_1(\mu_{j_{l-1}}+l-1)$ are smaller than $B_1(\mu_{j_l}+l)$ so there are total $(B_1(\mu_{j_l}+l)-(\mu_{j_l}+l+1)+l-1)$ elements smaller than $B_1(\mu_{j_l}+l)$ in $B_2$. So $B_1(\mu_{j_l}+l)$ is the $B_1(\mu_{j_l}+l)-(\mu_{j_l}+1)=(B_1(\mu_{j_l}+l)-(i_{j_l}-j_l))=r_l$-th smallest element in $B_2$. This implies that $1\leq r_1<\cdots<r_k\leq b$, so the sets $\{i_m |m \neq j_l\}$ and $\{n+a+b-r_l|l=1,\cdots,k\}$ are disjoint. Thus we have $A_2\subseteq \{0,\cdots,(n+a+b-1)\}$ with $|A_2|=a$. 

Denoting $\mu'=\lambda_{B_2}$, we have $\mu'_{r_l}+l=B_1(\mu_{j_l}+l)-r_l+1+l=i_{j_l}-j_l+l+1$ and $A_2\cap\{n+a,\cdots,n+a+b-1\}=\{n+a+b-r_1<\cdots<n+a+b-r_k\}$. Since there are $(j_l-l)$ elements smaller than $i_{j_l}$ in $A_2$, we have $i_{j_l}=A_2(i_{j_l}-j_l+l+1)=A_2(\mu'_{r_l}+l)$. We see $(A_1,B_1)=\psi_{n,b,a}((B_2,A_2))$.
\end{proof}
\begin{Ex}\label{exex}
For $n=1$, $a=3$ and $b=4$, consider $A_1=\{0,2,3\}$ and $B_1=\{2,4,5,7\}$ denoting $\mu=\lambda_{A_1}=(0,1,1)$. The above process changes the element $7=8-1 \in B_1$ to $B_1(\mu_1+1)=0$ and correspondingly change the element $0\in A_1$ to 7. Likewise, we change the element $5=8-3 \in B_1$ to the element $B_1(\mu_3+2)=3$ and change the element $3\in A_1$ to $5$. So we have $\psi_{1,3,4}((A_1,B_1))=(B_2,A_2)=(\{0,2,3,4\},\{2,5,7\})$. Since we have $\mu'=\lambda_{B_2}=(0,1,1,1)$, the element $7=8-1 \in A_2$ goes to $A_2(\mu'_1+1)=0$ and $0\in B_2$ goes to 7. Likewise the element $5=8-3 \in A_2$ goes to $A_2(\mu'_3+2)=3$ and $3\in B_2$ goes to 5. We see $\psi_{1,4,3}((B_2,A_2))=((A_1,B_1))$.
\end{Ex}
\begin{Rem}\label{ind}
Consider $B \subseteq \{0,\cdots,(n+a+b-1)\}$ with $|B|=b$ and $|B \cap \{0,\cdots,(n+b-1)\}|=b-k$. For each positive integer $l$, there are $(B(l)-l+1)$ elements in $B$ smaller than $B(l)$. And $B(l)$ is the unique integer with that property. Since we have $B(1),\cdots,B(n+k) \leq (n+b-1)$, for $C=B\cap\{0,\cdots,(n+b-2)\}$, we have $B(l)=C(l)$ for $l$ from 1 to $(n+k)$.  
\end{Rem}

\begin{Rem} \label{sym}
With the notation in Definition \ref{bijection1}, the weight $w_n(A_1,B_1)$ given in Definition \ref{weight} becomes 
$$w_n(A_1,B_1)=q^{\sum B_1-\sum B_2}(\prod\limits_{i\in A_1}X_i)(\prod\limits_{i\in B_2}Y_i).$$ 
Now let $\bar{w}_n(A_1,B_1)$ be the one obtained by exchanging $X_i \leftrightarrow Y_i$ from $w_n(A_1,B_1)$. Then we have 
\begin{equation*}
    \bar{w}_n(A_1,B_1)=q^{\sum B_1-\sum B_2}(\prod\limits_{i\in A_1}Y_i)(\prod\limits_{i\in B_2}X_i)=q^{\sum A_2-\sum A_1}(\prod\limits_{i\in B_2}X_i)(\prod\limits_{i\in A_1}Y_i)=w_n(B_2,A_2).
\end{equation*}
This gives
\begin{align*}
    \sum_{a+b=k}(\sum_{\substack{A_1 \subseteq \{0,\cdots,n+a-1\}\\|A_1|=a}}(\sum_{\substack{B_1 \subseteq \{0,\cdots,n+a+b-1\}\\|B_1|=b}}\bar{w}_n(A_1,B_1))) \\=
    \sum_{a+b=k}(\sum_{\substack{B_2 \subseteq \{0,\cdots,n+b-1\}\\|B_2|=b}}(\sum_{\substack{A_2 \subseteq \{0,\cdots,n+a+b-1\}\\|A_2|=a}}w_n(B_2,A_2))) \\=\sum_{a+b=k}(\sum_{\substack{A_1 \subseteq \{0,\cdots,n+a-1\}\\|A_1|=a}}(\sum_{\substack{B_1 \subseteq \{0,\cdots,n+a+b-1\}\\|B_1|=b}}w_n(A_1,B_1))).
\end{align*}
Thus the formula in Theorem \ref{main2} is invariant as a polynomial in $X_i$'s and $Y_i$'s under the exchange $X_i \leftrightarrow Y_i$. This was not true for the formula in Theorem \ref{main1}. 
\end{Rem}
\begin{Rem}\label{pos2}
When $\epsilon_2=0$, we have $Y_i=q^i \beta$. So the weight $w_n(A,B)$ simply becomes $(\prod\limits_{i \in A} X_i)(\prod\limits_{i \in B} Y_i)$. Then considering a directed graph in Figure \ref{graph2}, the formula for $g_{n+k,n}$ in Theorem \ref{main2} specializes to sum over weights of all paths from $u_n$ to $v_{n+k}$. When $\epsilon_1=0$, we have an analogous result by Remark \ref{sym}. 
\end{Rem}
\begin{figure}[ht] 
\begin{tikzpicture}
\filldraw[black] (5,0.2)circle (0.6pt);
\filldraw[black] (5,0.4)circle (0.6pt);
\filldraw[black] (5,0.6)circle (0.6pt);

\draw[->] (-2,0) -- (12,0);
\draw[->] (-2,-1.5) -- (12,-1.5);
\draw[->] (-2,-3) -- (12,-3);
\draw[->] (-2,-4.5) -- (12,-4.5);

\filldraw[black] (-2,0-4.5)circle (1pt) node[anchor=east] {$u_0$};
\filldraw[black] (-2,0-3)circle (1pt) node[anchor=east] {$u_1$};
\filldraw[black] (-2,0-1.5)circle (1pt) node[anchor=east] {$u_2$};
\filldraw[black] (-2,0)circle (1pt) node[anchor=east] {$u_3$};

\filldraw[black] (12,0-4.5)circle (1pt) node[anchor=west] {$v_0$};
\filldraw[black] (12,0-3)circle (1pt) node[anchor=west] {$v_1$};
\filldraw[black] (12,0-1.5)circle (1pt) node[anchor=west] {$v_2$};
\filldraw[black] (12,0)circle (1pt) node[anchor=west] {$v_3$};

\draw[->] (1.5,-4.5) -- (1.5,-3);
\filldraw[black] (1.5,-3.75)circle (0.0000001pt) node[anchor=east] {$X_0$};

\draw[->] (2.5,-3) -- (2.5,-1.5);
\filldraw[black] (2.5,-2.25)circle (0.0000001pt) node[anchor=east] {$X_1$};
\draw[->] (0.5,-3) -- (0.5,-1.5);
\filldraw[black] (0.5,-2.25)circle (0.0000001pt) node[anchor=east] {$X_0$};

\draw[->] (-0.5,-1.5) -- (-0.5,0);
\filldraw[black] (-0.5,-0.75)circle (0.0000001pt) node[anchor=east] {$X_0$};
\draw[->] (1.5,-1.5) -- (1.5,0);
\filldraw[black] (1.5,-0.75)circle (0.0000001pt) node[anchor=east] {$X_1$};
\draw[->] (3.5,-1.5) -- (3.5,0);
\filldraw[black] (3.5,-0.75)circle (0.0000001pt) node[anchor=east] {$X_2$};

\draw[->] (1.5+7,-4.5) -- (1.5+7,-3);
\filldraw[black] (1.5+7,-3.75)circle (0.0000001pt) node[anchor=east] {$Y_0$};

\draw[->] (2.5+7,-3) -- (2.5+7,-1.5);
\filldraw[black] (2.5+7,-2.25)circle (0.0000001pt) node[anchor=east] {$Y_1$};
\draw[->] (0.5+7,-3) -- (0.5+7,-1.5);
\filldraw[black] (0.5+7,-2.25)circle (0.0000001pt) node[anchor=east] {$Y_0$};

\draw[->] (-0.5+7,-1.5) -- (-0.5+7,0);
\filldraw[black] (-0.5+7,-0.75)circle (0.0000001pt) node[anchor=east] {$Y_0$};
\draw[->] (1.5+7,-1.5) -- (1.5+7,0);
\filldraw[black] (1.5+7,-0.75)circle (0.0000001pt) node[anchor=east] {$Y_1$};
\draw[->] (3.5+7,-1.5) -- (3.5+7,0);
\filldraw[black] (3.5+7,-0.75)circle (0.0000001pt) node[anchor=east] {$Y_2$};
\end{tikzpicture}\caption{The figure shows the weighted directed graph that gives rise to $g_{n+k,n}$ when $\epsilon_2=0$.}\label{graph2}
\end{figure}
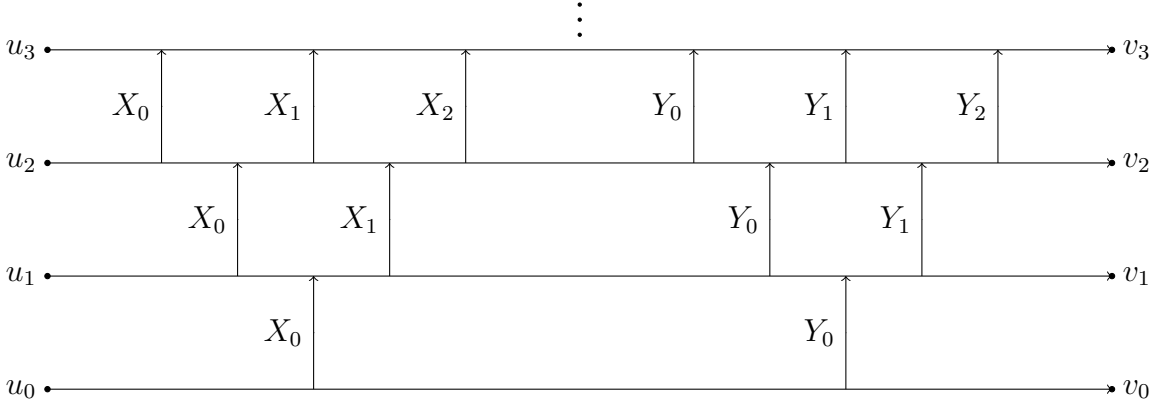

\subsection{Generalized \textit{q}-binomial coefficients}\label{gqb}
In this section, we prove Lemma \ref{ana} and Lemma \ref{t}. 

It follows from definitions (Definition \ref{weight}, Definitions \ref{wei}) that  
\begin{equation*}
    \sum_{\substack{B\subseteq\{0,\cdots,(n+a+b-1)\}\\|B|=b}}w_n(A,B)=(\prod\limits_{i\in A}X_i)M_{n}^{\lambda_A}(b)
\end{equation*} where $A\subseteq \{0,\cdots,(n+a-1)\}$ with $|A|=a$. Thus the formula in Theorem \ref{main2} can be rephrased as follows 
\begin{equation}\label{main22}
g_{n+k,k}=\sum_{a+b=k}\left(\sum_{\substack{A\subseteq\{0,\cdots,(n+a-1)\}\\|A|=a}}(\prod\limits_{i\in A}X_i)M_{n}^{\lambda_A}(b)\right).
\end{equation}

Note that we defined $M_{n}^{\mu}(b)$ for a weakly increasing composition $\mu$ possibly starting with (-1). To do that we  introduced a dummy variable $Y_{-1}=q^{-1}(\beta-\epsilon_1)$ which was defined accordingly to satisfy the recurrence $Y_{n+1}=q Y_{n}+\epsilon_2$. As $\lambda_A$ in \eqref{main22} consists of non-negative integers, we do not see $M_{n}^{\mu}(b)$'s such that $\mu$ starts with (-1) in \eqref{main22}. However, we will need them for the proof.

When $\mu$ consists of non-negative integers, we see that $M_n^{\mu}(b)$ is a polynomial in $Y_0,\cdots,Y_{n+b-1}$ with $\mathbb{Z}[q]$ coefficients. The next proposition computes the coefficient of each monomial.

\begin{Prop}\label{Coeff}Let $E\subseteq\{0,\cdots,n+b-1\}$ with $|E|=b$ and $\mu=(\nu_1^{e_1},\cdots,\nu_p^{e_p})$ such that $0 \leq \nu_1<\cdots<\nu_p\leq n$ with $e_i>0$. Denote the multiplicity of $\nu_i$ in $\lambda_E$ (can be possibly zero) by $f_i$ and write the corresponding elements of $E$ with $c_i,\cdots,(c_i+f_i-1)$. Then the coefficient of $\prod\limits_{i\in E} Y_i$ in $M_{n}^{\mu}(b)$ is given as follows
\begin{equation*}
    [\prod\limits_{i\in E} Y_i]M_{n}^{\mu}(b)=\sum_{k_1,\cdots,k_p}\left(\prod\limits_{i=1}^{p}(q^{k_i(d_i+k_i-1)}\binom{e_i}{k_i}_q\binom{f_i}{k_i}_q)\right),
\end{equation*} where $d_i=(n+b+\sum\limits_{j=i+1}^{p}e_j)-(c_i+f_i-1)$ and $k_i$ ranges from 0 to $min(e_i,f_i)$.
\end{Prop}
\begin{proof} To get a monomial $\prod\limits_{i\in E} Y_i$ in $M_{n}^{\mu}(b)$ we first pick $0\leq k_i \leq min(e_i,f_i)$ and then $k_i$ integers $(c_i+f_i-t_{i,1})<\cdots<(c_i+f_i-t_{i,k_i})$ in a set $\{c_i,\cdots,c_i+f_i-1\}$ and   $(n+b+\sum\limits_{j=i}^{p}e_j-u_{i,k_i})<\cdots<(n+b+\sum\limits_{j=i}^{p}e_j-u_{i,1})$ in a set $\{(n+b+\sum\limits_{j=i+1}^{p}e_j),\cdots,(n+b+\sum\limits_{j=i}^{p}e_j-1)\}$. Now for a set
\begin{equation*}
    B=(E-\bigcup\limits_{i=1}^{p}\{c_i+f_i-t_{i,1}<\cdots<c_i+f_i-t_{i,k_i}\})\cup(\bigcup\limits_{i=1}^{p}\{n+b+\sum\limits_{j=i}^{p}e_j-u_{i,k_i},\cdots,n+b+\sum\limits_{j=i}^{p}e_j-u_{i,1}\}),
\end{equation*}  we have
\begin{align*}
    B(\nu_1+j)=(c_1+f_1-t_{1,j})& \text{\hspace{2mm}for $j$ from 1 to $k_1$}\\
     B(\nu_2+k_1+j)=(c_2+f_2-t_{2,j})& \text{\hspace{2mm}for $j$ from 1 to $k_2$}\\
     \vdots&\\
     B(\nu_p+\sum\limits_{i=1}^{p-1} k_i+j)=(c_p+f_p-t_{p,j})& \text{\hspace{2mm}for $j$ from 1 to $k_p$},
\end{align*}
which gives
$m_{n}^{\mu}(B)=q^{\sum B-\sum E}\prod\limits_{i\in E} Y_i$. Summing over all possible such $B$ with fixed $k_i$'s, we get 
\begin{equation}\label{coeff}
    \prod\limits_{i=1}^{p}(q^{k_i(d_i+k_i-1)}\binom{e_i}{k_i}_q\binom{f_i}{k_i}_q)\prod\limits_{i\in E} Y_i.
\end{equation}
Summing \eqref{coeff} over all possible $k_i$'s gives the formula for the coefficient.
\end{proof}
\begin{Ex}
Let $E=\{0,1\}$ ($\lambda_E=(0,0)$) and $\mu=(0,0)$. Then the coefficient of $Y_0Y_1$ in $M_{2}^{\mu}(2)$ comes from the following
terms \begin{align*}
    &m_2^{\mu}(\{0,1\})=Y_0Y_1\\ &m_2^{\mu}(\{0,4\})=q^3Y_0Y_1, m_2^{\mu}(\{0,5\})=q^4Y_0Y_1, m_2^{\mu}(\{1,4\})=q^4Y_0Y_1, m_2^{\mu}(\{1,5\})=q^5Y_0Y_1\\&m_2^{\mu}(\{4,5\})=q^8Y_0Y_1.
\end{align*}
So we have 
\begin{equation*}
    [Y_0Y_1]M_2^{\mu}(2)=1+(q^3+2q^4+q^5)+q^8=\binom{2}{0}_q\binom{2}{0}_q+q^{1 \cdot 3}\binom{2}{1}_q\binom{2}{1}_q+q^{2\cdot4}\binom{2}{2}_q\binom{2}{2}_q.
\end{equation*}
\end{Ex}
Now we give a generalization of \eqref{zzzz}.
\begin{Lemma}\label{ana}
For a weakly increasing composition $\mu=(\mu_1,\cdots,\mu_a)$ with $0\leq\mu_1,\cdots\mu_a\leq n+1$, $n\geq0$ and $b\geq1$, the following identity holds
$$M^{\mu}_{n+1}(b)=Y_{n+a+b}M^{\mu}_{n+1}(b-1)+M^{\mu-1}_{n}(b),$$ where 
$\mu-1=(\mu_1-1,\cdots,\mu_a-1).$
\end{Lemma}
\begin{proof}
We have
\begin{align}\label{rf}
    &M^{\mu}_{n+1}(b)-(Y_{n+a+b}M^{\mu}_{n+1}(b-1)+M^{(\mu_1-1,\cdots,\mu_a-1)}_{n}(b))\\
    =&\sum_{\substack{B \subseteq \{0,\cdots,n+a+b\}\\|B|=b, (n+a+b)\in B}}(m^{\mu}_{n+1}(B)-Y_{n+a+b}m^{\mu}_{n+1}(B-\{n+a+b\}))\nonumber \\&+ \sum_{\substack{B \subseteq \{0,\cdots,n+a+b-1\}\\|B|=b}}(m^{\mu}_{n+1}(B)-m^{(\mu_1-1,\cdots,\mu_a-1)}_{n}(B)).\nonumber
\end{align}

We may assume $Y_n=\beta_1 q^n+\epsilon$ after the change of variables, $\beta_1=\beta-\frac{\epsilon_2}{1-q}$ and $\epsilon=\frac{\epsilon_2}{1-q}$ . For a set $C\subseteq\{0,\cdots,n+b-1\}$ such that $|C|=b-k$ and an increasing integer sequence $J=(j_1,\cdots,j_k)$ such that $1\leq j_1,\cdots,j_k\leq a+1$, we define
\begin{align*}
    &f_{C}^{J}(l)=\begin{cases*}
    q^{(n+b+a+1-j_l)-C(\mu_{j_l}+l)}Y_{C(\mu_{j_l}+l)} & if $j_l \neq a+1$ \\
    Y_{n+b} & if $j_l=a+1$
    \end{cases*}\\
    &d_{C}^{J}(k)=\begin{cases*}
    (n+b+a+1-j_l)-C(\mu_{j_l}+l) & if $j_l \neq a+1$ \\
    0 & if $j_l=a+1$
    \end{cases*}.
\end{align*}
We also define
\begin{align*}
    &\bar{f}_{C}^{J}(l)=\begin{cases*}
    =q^{(n+b+a+1-j_l)-C(\mu_{j_l-1}+l-1)}Y_{C(\mu_{j_l-1}+l-1)} & if $j_l \neq 1$ \\
    Y_{n+a+b} & if $j_l=1$
    \end{cases*}\\
    &\bar{d}_{C}^{J}(k)=\begin{cases*}
    (n+b+a+1-j_l)-C(\mu_{j_l-1}+l-1) & if $j_l \neq 1$ \\
    0 & if $j_l=1$
    \end{cases*}.
\end{align*}
Then by Definition \ref{wei} and Remark \ref{ind}, for $B=C\cup\{n+b+a+1-j_k<\cdots<n+b+a+1-j_1\}$, we have 
\begin{align*}
    &m^{\mu}_{n+1}(B)=(\prod\limits_{i \in C} Y_i)(f_{C}^{J}(k)\cdots f_{C}^{J}(1)) \\
    &Y_{n+b+a}m^{\mu}_{n+1}(B-\{n+b+a\})=(\prod\limits_{i \in C} Y_i)(\bar{f}_{C}^{J}(k)\cdots \bar{f}_{C}^{J}(1)) \text{\hspace{4mm}if $(n+b+a)\in B$}\\
    &m^{(\mu_1-1,\cdots,\mu_a-1)}_{n}(B)=(\prod\limits_{i \in C} Y_i)(\bar{f}_{C}^{J}(k)\cdots\bar{f}_{C}^{J}(1)) \text{\hspace{23mm}if $(n+b+a)\notin B$}.
\end{align*} So \eqref{rf} can be written as follows
\begin{equation}\label{zz}\sum_{k\geq0}(\sum_{\substack{C\subseteq\{0,\cdots,n+b-1\}\\|C|=b-k}}(\prod_{i \in C} Y_i)(\sum_{\substack{J=(j_1,\cdots,j_k)\\1\leq j_1<\cdots<j_k\leq a+1}}(f_{C}^{J}(k)\cdots f_{C}^{J}(1)-\bar{f}_{C}^{J}(k)\cdots \bar{f}_{C}^{J}(1)))).
\end{equation}
We first rewrite the quantity
\begin{equation}\label{quan}
    \sum_{\substack{J=(j_1,\cdots,j_k)\\1\leq j_1<\cdots<j_k\leq a+1}}(f_{C}^{J}(k)\cdots f_{C}^{J}(1)-\bar{f}_{C}^{J}(k)\cdots \bar{f}_{C}^{J}(1))
\end{equation} as follows \begin{align}\label{eq11}
  &\sum_{\substack{J=(j_1,\cdots,j_k)\\1\leq j_1<\cdots<j_k\leq a+1}}(f_{C}^{J}(k)\cdots f_{C}^{J}(1)-\bar{f}_{C}^{J}(k)\cdots \bar{f}_{C}^{J}(1))\nonumber\\=&\sum_{\substack{J=(j_1,\cdots,j_k)\\1\leq j_1<\cdots<j_k\leq a+1}}(\sum_{l=1}^{k}\bar{f}_{C}^{J}(k)\cdots\bar{f}_{C}^{J}(k-l+2)\left(f_{C}^{J}(k-l+1)-\bar{f}_{C}^{J}(k-l+1)\right)f_{C}^{J}(k-l)\cdots f_{C}^{J}(1))\nonumber\\=&\sum_{\substack{J=(j_1,\cdots,j_k)\\1\leq j_1<\cdots<j_k\leq a+1}}(\sum_{l=1}^{k}\bar{f}_{C}^{J}(k)\cdots\bar{f}_{C}^{J}(k-l+2)\left(q^{d_{C}^{J}(k-l+1)}\epsilon-q^{\bar{d}_{C}^{J}(k-l+1)}\epsilon\right)f_{C}^{J}(k-l)\cdots f_{C}^{J}(1))  \nonumber\\=&\epsilon\sum_{l=1}^{k}(\sum_{\substack{J=(j_1,\cdots,j_k)\\1\leq j_1<\cdots<j_k\leq a+1}}\bar{f}_{C}^{J}(k)\cdots\bar{f}_{C}^{J}(k-l+2)(q^{d_{C}^{J}(k-l+1)})f_{C}^{J}(k-l)\cdots f_{C}^{J}(1)) \\&-\epsilon\sum_{l=1}^{k}(\sum_{\substack{J=(j_1,\cdots,j_k)\\1\leq j_1<\cdots<j_k\leq a+1}}\bar{f}_{C}^{J}(k)\cdots\bar{f}_{C}^{J}(k-l+2)(q^{\bar{d}_{C}^{J}(k-l+1)})f_{C}^{J}(k-l)\cdots f_{C}^{J}(1)). \nonumber     
\end{align}
For $J=(j_1,\cdots,j_k)$ such that $j_{k-l+1}=j_{k-l}+1$, we have$$\bar{f}_{C}^{J}(k-l+1)(q^{d_{C}^{J}(k-l)})=(q^{\bar{d}_{C}^{J}(k-l+1)})f_{C}^{J}(k-l)$$ since $C(\mu_{(j_{k-l+1}-1)}+k-l)=C(\mu_{j_{k-l}}+k-l).$ So we have \begin{align*}
    \bar{f}_{C}^{J}(k)\cdots\bar{f}_{C}^{J}(k-l+1)(q^{d_{C}^{J}(k-l)})f_{C}^{J}(k-l-1)\cdots f_{C}^{J}(1) \\=
    \bar{f}_{C}^{J}(k)\cdots\bar{f}_{C}^{J}(k-l+2)(q^{\bar{d}_{C}^{J}(k-l+1)})f_{C}^{J}(k-l)\cdots f_{C}^{J}(1).
\end{align*}
Applying this cancellation to \eqref{eq11}, it becomes \begin{align}\label{quan1}
    &\epsilon(\sum_{\substack{J=(j_1,\cdots,j_k)\\1\leq j_1<\cdots<j_k\leq {a+1}}} q^{d^{J}_{C}}(k)f_{C}^{J}(k-1)\cdots f_{C}^{J}(1))\\&+\epsilon\sum_{l=2}^{k}(\sum_{\substack{J=(j_1,\cdots,j_k)^{l-1}\\1\leq j_1<\cdots<j_k\leq a}}\bar{f}_{C}^{J}(k)\cdots\bar{f}_{C}^{J}(k-l+2)(q^{d_{C}^{J}(k-l+1)})f_{C}^{J}(k-l)\cdots f_{C}^{J}(1))\nonumber \\&-\epsilon\sum_{l=1}^{k-1}(\sum_{\substack{J=(j_1,\cdots,j_k)^{l}\\1\leq j_1<\cdots<j_k\leq a}}\bar{f}_{C}^{J}(k)\cdots\bar{f}_{C}^{J}(k-l+2)(q^{\bar{d}_{C}^{J}(k-l+1)})f_{C}^{J}(k-l)\cdots f_{C}^{J}(1))\nonumber\\&
    -\epsilon(\sum_{\substack{J=(j_1,\cdots,j_k)\\1\leq j_1<\cdots<j_k\leq {a+1}}} \bar{f}_{C}^{J}(k)\cdots \bar{f}_{C}^{J}(2) q^{\bar{d}^{J}_{C}}(1)),\nonumber
\end{align}
where $(j_1,\cdots,j_k)^{l}:=(j_1,\cdots,j_{k-l},(j_{k-l+1}+1),\cdots,j_k+1).$ We regard  $(j_1,\cdots,j_k)^{0}=(j_1,\cdots,j_k)$ by convention. Now define 
\begin{align*}
    W_{C}^{J}(l)=\bar{f}_{C}^{J'}(k)\cdots\bar{f}_{C}^{J'}(k-l+2)(q^{d_{C}^{J'}(k-l+1)})f_{C}^{J'}(k-l)\cdots f_{C}^{J'}(1),
\end{align*} where $J'=J^{l-1}$ and
\begin{align*}
   \bar{W}_{C}^{J}(l)=\bar{f}_{C}^{J'}(k)\cdots\bar{f}_{C}^{J'}(k-l+2)(q^{\bar{d}_{C}^{J'}(k-l+1)})f_{C}^{J'}(k-l)\cdots f_{C}^{J'}(1)),
\end{align*} where $J'=J^{l}$. We also define 
\begin{align*}
    &V_{C}^{J}=f_{C}^{J}(k)\cdots f_{C}^{J}(1),
    \\
     &\bar{V}_{C}^{J}=\bar{f}_{C}^{J'}(k+1)\cdots f_{C}^{J'}(2),
\end{align*}
where $J'=(1,J^k)$ and $k$ is a length of $J$. Since we have
\begin{align*}
    &\sum_{\substack{J=(j_1,\cdots,j_k)\\1\leq j_1<\cdots<j_k\leq {a+1}}} q^{d^{J}_{C}}(k)f_{C}^{J}(k-1)\cdots f_{C}^{J}(1)\\=&\sum_{\substack{J=(j_1,\cdots,j_k)\\1\leq j_1<\cdots<j_k\leq a}} q^{d^{J}_{C}}(k)f_{C}^{J}(k-1)\cdots f_{C}^{J}(1)+\sum_{\substack{J=(j_1,\cdots,j_k)\\1\leq j_1<\cdots<j_{k}=a+1}} f_{C}^{J}(k-1)\cdots f_{C}^{J}(1)\\
    =&\sum_{\substack{J=(j_1,\cdots,j_k)\\1\leq j_1<\cdots<j_k\leq a}} W_{C}^{J}(1)+\sum_{\substack{J'=(j_1,\cdots,j_{k-1})\\1\leq j_1<\cdots<j_{k-1}\leq a}} V_{C}^{J'},\\
    &\sum_{\substack{J=(j_1,\cdots,j_k)\\1\leq j_1<\cdots<j_k\leq {a+1}}} \bar{f}_{C}^{J}(k)\cdots \bar{f}_{C}^{J}(2) q^{\bar{d}^{J}_{C}}(1)\\=&\sum_{\substack{J=(j_1,\cdots,j_k)^k\\1\leq j_1<\cdots<j_k\leq a}} \bar{f}_{C}^{J}(k)\cdots \bar{f}_{C}^{J}(2) q^{\bar{d}^{J}_{C}}(1)+\sum_{\substack{J=(j_1,\cdots,j_k)\\1= j_1<\cdots<j_{k}\leq a+1}} \bar{f}_{C}^{J}(k)\cdots \bar{f}_{C}^{J}(2)
    \\=&\sum_{\substack{J=(j_1,\cdots,j_k)\\1\leq j_1<\cdots<j_k\leq a}} \bar{W}_{C}^{J}(k)+\sum_{\substack{J'=(j_1,\cdots,j_{k-1})\\1\leq j_1<\cdots<j_{k-1}\leq a}} \bar{V}_{C}^{J'},
\end{align*}
the quantity \eqref{quan1} can be written as 
\begin{align*}
    \epsilon(\sum_{\substack{J=(j_1,\cdots,j_k)\\1\leq j_1<\cdots<j_k\leq a}}\sum_{l=1}^{k}(W_{C}^{J}(l)-\bar{W}_{C}^{J}(l)))+\epsilon(\sum_{\substack{J'=(j_1,\cdots,j_{k-1})\\1\leq j_1<\cdots<j_{k-1}=a}} (V_{C}^{J}-\bar{V}_{C}^{J}).
\end{align*}
Plugging this to \eqref{zz} and dividing with $\epsilon$ yields
\begin{align}\label{quan2}
&\sum_{k\geq0}(\sum_{\substack{C\subseteq\{0,\cdots,n+b-1\}\\|C|=b-k}}(\prod_{i \in C} Y_i)(\sum_{\substack{J=(j_1,\cdots,j_k)\\1\leq j_1<\cdots<j_k\leq a}}\sum_{l=1}^{k}(W_{C}^{J}(l)-\bar{W}_{C}^{J}(l))))\nonumber\\&+
\sum_{k\geq0}(\sum_{\substack{D\subseteq\{0,\cdots,n+b-1\}\\|D|=b-k-1}}(\prod_{i \in D} Y_i)(\sum_{\substack{J=(j_1,\cdots,j_k)\\1\leq j_1<\cdots<j_k\leq a}}(V_D^{J}-\bar{V}_D^{J}))\nonumber\\
=&\sum_{\substack{J=(j_1,\cdots,j_k)\\1\leq j_1<\cdots<j_k\leq a}}(\sum_{\substack{C\subseteq\{0,\cdots,n+b-1\}\\|C|=b-k}}(\prod_{i \in C} Y_i)\sum_{l=1}^{k}(W_{C}^{J}(l)-\bar{W}_{C}^{J}(l))\\&+
\sum_{\substack{D\subseteq\{0,\cdots,n+b-1\}\\|D|=b-k-1}}(\prod_{i \in D} Y_i)((V_D^{J}-\bar{V}_D^{J})))\nonumber
.
\end{align}

We will show that the quantity 
\begin{equation}\label{quan3}\sum_{\substack{C\subseteq\{0,\cdots,n+b-1\}\\|C|=b-k}}(\prod_{i \in C} Y_i)\sum_{l=1}^{k}(W_{C}^{J}(l)-\bar{W}_{C}^{J}(l))\\+
\sum_{\substack{D\subseteq\{0,\cdots,n+b-1\}\\|D|=b-k-1}}(\prod_{i \in D} Y_i)((V_D^{J}-\bar{V}_D^{J}))\end{equation} vanishes as a polynomial in $Y_i$'s for every $J=(j_1,\cdots,j_k)$ with $1\leq j_1<\cdots<j_k\leq a$. Then it shows that \eqref{quan2} vanishes. We will denote $(\mu_{j_1},\cdots,\mu_{j_k})$ as $(\nu_{1}^{e_1},\cdots,\nu_{p}^{e_p})$ such that $\nu_1<\cdots<\nu_p$ and $e_i>0$. Then we have $W_{l}=\bar{W}_{l-1}$ if 
$l=(\sum\limits_{i=r+1}^{p} e_i+2),\cdots,(\sum\limits_{i=r}^{p} e_i)$ for $r=1,\cdots,p$. So \eqref{quan3} becomes 
\begin{align}\label{qu4}\sum_{\substack{C\subseteq\{0,\cdots,n+b-1\}\\|C|=b-k}}(\prod_{i \in C} Y_i)\sum_{r=1}^{p}(W_{C}^{J}(\sum\limits_{i=r+1}^{p} e_i+1)-\bar{W}_{C}^{J}(\sum\limits_{i=r}^{p} e_i))\\+
\sum_{\substack{D\subseteq\{0,\cdots,n+b-1\}\\|D|=b-k-1}}(\prod_{i \in D} Y_i)((V_D^{J}-\bar{V}_D^{J})).\nonumber\end{align}
Now we will take the coefficient of the monomial $\prod\limits_{i\in E} Y_i$ in the quantity. Denote the multiplicity of $\nu_i$ in $\lambda_E$ with $f_i$ (can be possibly zero) and corresponding elements of $E$ with $c_i,\cdots,(c_i+f_i-1)$. Let $d_{i,h}=(n+b+a-j_s)-(c_i+f_i)$ where $s=\sum\limits_{l=1}^{i-1}e_l+h$ for $1\leq h\leq e_i$. To get a monomial $\prod\limits_{i\in E} Y_i$ in $(\prod\limits_{i \in C} Y_i)(W_{C}^{J}(\sum\limits_{i=r+1}^{p} e_i+1)-\bar{W}_{C}^{J}(\sum\limits_{i=r}^{p} e_i))$ we should take $C$ as 
$$C=E-(\bigcup_{\substack{i=1\\i \neq r}}^{p} \{c_i+f_i-t_{i,e_i}<\cdots<c_i+f_i-t_{i,1}\})-\{c_r+f_r-t_{r,e_r-1}<\cdots<c_r+f_r-t_{r,1}\}$$
such that $1\leq t_{i,h} \leq f_i$. Then for such $C$ we have ($J'=J^{\sum\limits_{i=r+1}^{p}e_i}$)
\begin{align*}
    &f_C^{J'}(\sum\limits_{l=1}^{i-1}e_l+h)=q^{d_{i,h}+t_{i,h}+1}Y_{c_i+f_i-t_{i,h}} \text{\hspace{2mm}for $1\leq i\leq r-1$ and $1\leq h\leq e_i$}\\
     &f_C^{J'}(\sum\limits_{l=1}^{r-1}e_l+h)=q^{d_{r,h}+t_{r,h}+1}Y_{c_r+f_r-t_{r,h}} \text{\hspace{2mm}for $1\leq h\leq e_r-1$}\\
    &f_C^{J'}(\sum\limits_{l=1}^{r}e_l)=q^{d_{r,e_r}+1}Y_{c_r+f_r}\rightarrow d_C^{J'}(\sum\limits_{l=1}^{r}e_l)=d_{r,e_r}+1\\
     &\bar{f}_C^{J'}(\sum\limits_{l=1}^{i-1}e_l+h)=q^{d_{i,h}+t_{i,h}}Y_{c_i+f_i-t_{i,h}} \text{\hspace{2mm}for $r+1\leq i\leq p$ and $1\leq h\leq e_i$},
\end{align*}
which gives 
\begin{equation*} (\prod_{i \in C} Y_i)W_{C}^{J}(\sum\limits_{i=r+1}^{p} e_i+1)=q^{\sum\limits_{i,h}d_{i,h}}q^{\sum\limits_{l=1}^{r}e_l}q^{\sum\limits_{i,h} t_{i,h}}\prod\limits_{i\in E} Y_i.\end{equation*}
And we have ($J'=J^{\sum\limits_{i=r}^{p}e_i}$)
\begin{align*}
 &\bar{f}_C^{J'}(\sum\limits_{l=1}^{r-1}e_l+1)=q^{d_{r,1}+f_r+1}Y_{c_r-1}\rightarrow d_C^{J'}(\sum\limits_{l=1}^{r}e_l)=d_{r,1}+f_r+1\\
     &\bar{f}_C^{J'}(\sum\limits_{l=1}^{r-1}e_l+h+1)=q^{d_{r,h}+t_{r,h}}Y_{c_r+f_r-t_{r,h}} \text{\hspace{2mm}for $1\leq h\leq e_r-1$},
\end{align*}
which gives 
\begin{equation*} (\prod_{i \in C} Y_i)\bar{W}_{C}^{J}(\sum\limits_{i=r}^{p} e_i)=q^{\sum\limits_{i,h}d_{i,h}}q^{\sum\limits_{l=1}^{r-1}e_l}q^{(\sum\limits_{i,h} t_{i,h})+f_r+1}\prod\limits_{i\in E} Y_i.\end{equation*}
So we have
\begin{align*}
  &(\prod_{i \in C} Y_i)(W_{C}^{J}(\sum\limits_{i=r+1}^{p} e_i+1)-\bar{W}_{C}^{J}(\sum\limits_{i=r}^{p} e_i))\\=&q^{\sum\limits_{i,h}d_{i,h}}q^{\sum\limits_{l=1}^{r-1}e_l}q^{\sum\limits_{i,h} t_{i,h}}(q^{e_r}-q^{f_r+1})\prod\limits_{i\in E} Y_i.  
\end{align*}
Summing over all possible $t_{i,h}$'s we have  
\begin{align*}
&[\prod\limits_{i\in E} Y_i]\left(\sum\limits_{\substack{C\subseteq\{0,\cdots,n+b-1\}\\|C|=b-k}}(\prod_{i \in C} Y_i)(W_{C}^{J}(\sum\limits_{i=r+1}^{p} e_i+1)-\bar{W}_{C}^{J}(\sum\limits_{i=r}^{p} e_i))\right)\\=&
   q^{\sum\limits_{i,h}d_{i,h}}q^{\sum\limits_{l=1}^{r-1}e_l}q^{\binom{e_r}{2}}\binom{f_r}{e_r-1}_q (\prod\limits_{\substack{l=1 \\ l \neq r}}^{p}(q^{\binom{e_l+1}{2}}\binom{f_l}{e_l}_q))(q^{e_r}-q^{f_r+1})
   \\=& q^{\sum\limits_{i,h}d_{i,h}}q^{\sum\limits_{l=1}^{r-1}e_l} (\prod\limits_{l=1}^{p}(q^{\binom{e_l+1}{2}}\binom{f_l}{e_l}_q))(1-q^{e_r})= q^{\sum\limits_{i,h}d_{i,h}} (\prod\limits_{l=1}^{p}(q^{\binom{e_l+1}{2}}\binom{f_l}{e_l}_q))(q^{\sum\limits_{l=1}^{r-1}e_l}-q^{\sum\limits_{l=1}^{r}e_l}).
\end{align*}
Now summing over all $r=1,\cdots,p$ we have
\begin{align}\label{q3}
    &[\prod\limits_{i\in E} Y_i]\left(\sum\limits_{\substack{C\subseteq\{0,\cdots,n+b-1\}\\|C|=b-k}}(\prod_{i \in C} Y_i)\sum\limits_{r=1}^{p}(W_{C}^{J}(\sum\limits_{i=r+1}^{p} e_i+1)-\bar{W}_{C}^{J}(\sum\limits_{i=r}^{p} e_i))\right)\\
    =&q^{\sum\limits_{i,h}d_{i,h}} (\prod\limits_{l=1}^{p}(q^{\binom{e_l+1}{2}}\binom{f_l}{e_l}_q))(1-q^{\sum\limits_{l=1}^{p}e_l}).\nonumber
\end{align}
To get a monomial $\prod\limits_{i\in E} Y_i$ in $(\prod\limits_{i \in D} Y_i)((V_D^{J}-\bar{V}_D^{J})$, we should take $D$ as 
$$D=E-\bigcup_{\substack{i=1\\}}^{p} \{c_i+f_i-t_{i,e_i}<\cdots<c_i+m_i-t_{i,1}\}$$
such that $1\leq t_{i,h} \leq f_i$. Then for such $D$ we have 
\begin{align*}
    &f_D^{J}(\sum\limits_{l=1}^{i-1}e_l+h)=q^{d_{i,h}+t_{i,h}+1}Y_{c_i+f_i-t_{i,h}} \text{\hspace{2mm}for $1\leq i\leq p$ and $1\leq h\leq e_i$}\\
    &\bar{f}_D^{J'}(\sum\limits_{l=1}^{i-1}e_l+h+1)=q^{d_{i,h}+t_{i,h}}Y_{c_i+f_i-t_{i,h}} \text{\hspace{2mm}for $1\leq i\leq p$ and $1\leq h\leq e_i$},
\end{align*}
where $J'=(0,J^k)$. So we have
\begin{align*}
    &(\prod_{i \in D} Y_i)V_{C}^{J}=q^{\sum\limits_{i,h}d_{i,h}}q^{\sum\limits_{l=1}^{p}e_l}q^{\sum\limits_{i,h} t_{i,h}}\prod\limits_{i\in E} Y_i
    \\
    &(\prod_{i \in D} Y_i)\bar{V}_{C}^{J}=q^{\sum\limits_{i,h}d_{i,h}}q^{\sum\limits_{i,h} t_{i,h}}\prod\limits_{i\in E} Y_i,
\end{align*} 
which gives 
\begin{align*}
  (\prod_{i \in D} Y_i)(V_{C}^{J}-\bar{V}_{C}^{J})=q^{\sum\limits_{i,h}d_{i,h}}q^{\sum\limits_{i,h} t_{i,h}}(q^{\sum\limits_{l=1}^{p}e_l}-1)\prod\limits_{i\in E} Y_i.  
\end{align*}
Summing over all possible $t_{i,h}$'s we have
\begin{align}\label{q5}
    &[\prod\limits_{i\in E} Y_i]\left(\sum_{\substack{D\subseteq \{0,\cdots,n+b-1\}\\|D|=b-k-1}}(\prod\limits_{i \in D} Y_i)((V_D^{J}-\bar{V}_D^{J})\right)\\=&q^{\sum\limits_{i,h}d_{i,h}} (\prod\limits_{l=1}^{p}(q^{\binom{e_l+1}{2}}\binom{f_l}{e_l}_q))(q^{\sum\limits_{l=1}^{p}e_l}-1).\nonumber
\end{align}
Adding \eqref{q3} and \eqref{q5}, we see that the coefficient vanishes. 
\end{proof} 
Next, we generalize \eqref{q1q1} (Lemma \ref{t}). Before stating and proving the generalization, we prepare with a definition and a lemma. 
\begin{Def}
For $\nu=(\tau_{1}^{e_1},\cdots,\tau_{p}^{e_p})$ such that $-1\leq\tau_{1}<\cdots<\tau_{p}$ and $e_i>0$, we define 
\begin{align*}
\nu(i) =
    \begin{cases*}
    (i,\tau_{1}^{e_1},\cdots,\tau_{p}^{e_p}) & if $-1\leq i\leq \tau_{1}-1$ \\
      (\tau_{1}^{e_1},\cdots,\tau_{l}^{e_l},i,\tau_{l+1}^{e_{l+1}},\cdots,\tau_{p}^{e_p}) & if $\tau_l\leq i\leq \tau_{l+1}-1$ \\
      (\tau_{1}^{e_1},\cdots,\tau_{p}^{e_p},i) & if $ \tau_{p}\leq i$
    \end{cases*}.
\end{align*}
\end{Def}
For example, if $\nu=(-1,-1,1)$, we have $\nu(-1)=(-1,-1,-1,1)$, $\nu(0)=(-1,-1,0,1)$, $\nu(1)=(-1,-1,1,1)$ and $\nu(2)=(-1,-1,1,2)$.

\begin{Lemma}\label{xx}
For $\nu=(\nu_1,\cdots,\nu_{a-1})$ such that $0\leq\nu_1\leq\cdots\leq\nu_{a-1}\leq n+1$, we have
$$M_{n}^{(-1,\nu-1)}(b)=q^{n+a+b}Y_{-1}M_{n+1}^{\nu}(b-1)+M_{n}^{\nu-1}(b).$$
\end{Lemma}
\begin{proof}
We have $m_{n}^{(-1,\nu-1)}(B\cup\{n+a+b-1\})=q^{n+a+b}Y_{-1}m_{n+1}^{\nu}(B)$ for $B\subseteq\{0,\cdots,n+a+b-2\}$ with $|B|=b-1$ and $m_{n}^{(-1,\nu-1)}(B)=m_{n}^{\nu-1}(B)$ for $B\subseteq\{0,\cdots,n+a+b-2\}$ with $|B|=b$.
\end{proof}
\begin{Lemma}\label{t}
For $\nu=(\tau_1^{e_1},\cdots,\tau_p^{e_p})$ such that $0\leq\tau_1<\cdots<\tau_p\leq n$ and $e_i>0$, we have
\begin{align}\label{first}
    [n+b+1+\sum\limits_{i=1}^{p}e_i]_q M_{n}^{\nu}(b)=&(\sum\limits_{i=1}^{p}q^{(\tau_i+\sum\limits_{l=1}^{i-1}e_l)}[e_i+1]_q M_{n}^{\nu(\tau_i)}(b))+\sum\limits_{l=0}^{\tau_1-1}q^{l}M_{n}^{\nu(l)}(b)\\&+\sum\limits_{i=1}^{p-1}(\sum\limits_{l=\tau_i+1}^{\tau_{i+1}-1}q^{(\sum\limits_{j=1}^{i}e_j)+l}M_{n}^{\nu(l)}(b))+\sum\limits_{l=\tau_p+1}^{n}q^{(\sum\limits_{j=1}^{p}e_j)+l} M_{n}^{\nu(l)}(b)\nonumber .
\end{align}
For $\nu=((-1)^{e_1},\tau_2^{e_2},\cdots,\tau_p^{e_p})$ such that $0\leq\tau_2<\cdots<\tau_p\leq n$ and $e_i>0$, we have
\begin{align}\label{second}
    [n+b+1+\sum\limits_{i=1}^{p}e_i]_q M_{n}^{\nu}(b)=&(\sum\limits_{i=2}^{p}q^{(\tau_i+\sum\limits_{l=1}^{i-1}e_l)}[e_i+1]_q M_{n}^{\nu(\tau_i)}(b))+[e_1]_q M_{n}^{\nu(-1)}(b)\\&+\sum\limits_{i=1}^{p-1}(\sum\limits_{l=\tau_i+1}^{\tau_{i+1}-1}q^{(\sum\limits_{j=1}^{i}e_j)+l}M_{n}^{\nu(l)}(b))+\sum\limits_{l=\tau_p+1}^{n}q^{(\sum\limits_{j=1}^{p}e_j)+l} M_{n}^{\nu(l)}(b)\nonumber .
\end{align}
\end{Lemma}
\begin{proof}
We first prove \eqref{first}. We will show that the equality holds as a polynomial in $Y_i$'s. Consider a length $b$ integer vector $\mu=(0^{f_0},\cdots,n^{f_n})$ such that $f_i\geq 0$, then we will compare the coefficients of $Y_\mu$($:=\prod\limits_{i=1}^{b}Y_{\mu_i+i-1}$) in both sides of \eqref{first}. First define the following \begin{align*}
g^{(k_1,\cdots,k_p)}(i) =
    \begin{cases*}
    q^{i}+q^{n+(\sum\limits_{j=i+1}^{n}f_j)+(\sum\limits_{j=1}^{p}e_j)+1}[f_i]_q & if $0\leq i\leq \tau_{1}-1$ \\
      q^{i+(\sum\limits_{j=1}^{l}k_j+(\sum\limits_{j=1}^{l}e_j)}+q^{n+(\sum\limits_{j=i+1}^{n}f_j)+(\sum\limits_{j=1}^{p}e_j)+(\sum\limits_{j=1}^{l}k_j)+1}[f_i]_q & if $\tau_l+1\leq i\leq \tau_{l+1}-1$ \\
      q^{i+(\sum\limits_{j=1}^{p}k_j+(\sum\limits_{j=1}^{p}e_j)}+q^{n+(\sum\limits_{j=i+1}^{n}f_j)+(\sum\limits_{j=1}^{p}e_j)+(\sum\limits_{j=1}^{p}k_j)+1}[f_i]_q  & if $ \tau_{p}+1\leq i\leq n$
      \\
      q^{\tau_l+(\sum\limits_{j=1}^{l-1}e_j)+(\sum\limits_{j=1}^{l-1}k_j)}[e_l+k_l+1]_q\\+q^{n+(\sum\limits_{j=\tau_l+1}^{n}f_j)+(\sum\limits_{j=1}^{p}e_j)+(\sum\limits_{j=1}^{l}k_j)+1}[f_{\tau_l}-k_l]_q  & if $ i=\tau_l$
    \end{cases*}
\end{align*}
\begin{align*}
d(i) =
    \begin{cases*}
    n+b+(\sum\limits_{j=1}^{p}e_j)-(i+(\sum\limits_{j=0}^{i}f_j)-1) & if $0\leq i\leq \tau_{1}-1$ \\
      n+b+(\sum\limits_{j=l+1}^{p}e_j)-(i+(\sum\limits_{j=0}^{i}f_j)-1) & if $\tau_l\leq i\leq \tau_{l+1}-1$ \\
     n+b-(i+(\sum\limits_{j=0}^{i}f_j)-1)  & if $ \tau_{p}\leq i\leq n$
      \end{cases*},
\end{align*}
where $k_i$ is an integer from 0 to $min(e_i,f_{\tau_i})$. Then we have \begin{equation}\label{whw}
    \sum\limits_{i=0}^{n}g^{(k_1,\cdots,k_p)}(i)= [n+b+1+\sum\limits_{i=1}^{p}e_i]_q.
\end{equation} 
Next we will define the following
\begin{align*}
    S^{\nu}(k_1,\cdots,k_p)=q^{\sum\limits_{i=1}^{p}k_i(d(\tau_i)+k_i-1)}\prod_{i=1}^{p}(\binom{e_i}{k_i}_q \binom{f_{\tau_i}}{k_i}_q)
\end{align*}
\begin{align*}
S^{\nu(i)}(k_1,\cdots,k_p) =
    \begin{cases*}
    q^{\sum\limits_{i=1}^{p}k_i(d(\tau_i)+k_i-1)}(\prod\limits_{i=1}^{p}(\binom{e_i}{k_i}_q \binom{f_{\tau_i}}{k_i}_q))(1+q^{d(i)}[f_i]_q)
     & if $0\leq i\leq \tau_{1}-1$ \\
      q^{\sum\limits_{i=1}^{l}k_i}q^{\sum\limits_{i=1}^{p}k_i(d(\tau_i)+k_i-1)}(\prod\limits_{i=1}^{p}(\binom{e_i}{k_i}_q \binom{f_{\tau_i}}{k_i}_q))(1+q^{d(i)}[f_i]_q) & if $\tau_l+1\leq i\leq \tau_{l+1}-1$ \\
      q^{\sum\limits_{i=1}^{p}k_i}q^{\sum\limits_{i=1}^{p}k_i(d(\tau_i)+k_i-1)}(\prod\limits_{i=1}^{p}(\binom{e_i}{k_i}_q \binom{f_{\tau_i}}{k_i}_q))(1+q^{d(i)}[f_i]_q) & if $ \tau_{p}+1\leq i\leq n$
      \\
       q^{\sum\limits_{i=1}^{l-1}k_i}q^{\sum\limits_{i=1}^{p}k_i(d(\tau_i)+k_i-1)}(\prod\limits_{\substack{i=1\\i \neq l}}^{p}(\binom{e_i}{k_i}_q \binom{f_{\tau_i}}{k_i}_q))\binom{e_l+1}{k_l}_q \binom{f_{\tau_l}}{k_l}_q& if $ i=\tau_l$
    \end{cases*}.
\end{align*}
By Proposition \ref{coeff}, we have
\begin{align*}
    [Y_{\mu}](M_{n}^{\nu}(b))=\sum_{k_1,\cdots,k_p}S^{\nu}(k_1,\cdots,k_p)\\
    [Y_{\mu}](M_{n}^{\nu(i)}(b))=\sum_{k_1,\cdots,k_p}S^{\nu(i)}(k_1,\cdots,k_p).
\end{align*}
Taking the coefficient of $Y_\mu$ in the left hand side of \eqref{first} and using \eqref{whw} gives
\begin{align*}
    (\sum_{k_1,\cdots,k_p}S^{\nu}(k_1,\cdots,k_p))[n+b+1+\sum\limits_{i=1}^{p}e_i]_q= \sum\limits_{i=0}^{n}(\sum_{k_1,\cdots,k_p}S^{\nu}(k_1,\cdots,k_p))g^{(k_1,\cdots,k_p)}(i)).
\end{align*}
It is straightforward to check the following
\begin{align*}
S^{\nu}(k_1,\cdots,k_p)g^{(k_1,\cdots,k_p)}(i) =
    \begin{cases*}
   q^{i}S^{\nu(i)}(k_1,\cdots,k_p) & if $0\leq i\leq \tau_{1}-1$ \\
       q^{(\sum\limits_{j=1}^{l}e_j)+i}S^{\nu(i)}(k_1,\cdots,k_p) & if $\tau_l+1\leq i\leq \tau_{l+1}-1$ \\
      q^{(\sum\limits_{j=1}^{p}e_j)+i}S^{\nu(i)}(k_1,\cdots,k_p)  & if $ \tau_{p}+1\leq i\leq n$
      \end{cases*},
\end{align*}
so it suffices to prove \begin{equation}\label{ahffk}
    \sum\limits_{k_1,\cdots,k_p}S^{\nu}(k_1,\cdots,k_p))g^{(k_1,\cdots,k_p)}(\tau_l)=[Y_{\mu}](q^{\tau_l+(\sum\limits_{i=1}^{l-1}e_i)}[e_l+1]_qM_{n}^{\nu(\tau_l)}(b)).
\end{equation}
We first claim the following identity
\begin{align}\label{tlqkf}
    q^{\tau_l+(\sum\limits_{j=1}^{l-1}e_j)}[e_l+1]_q S^{\nu(\tau_l)}(k_1,\cdots,k_p)=S^{\nu}(k_1,\cdots,k_p)( q^{\tau_l+(\sum\limits_{j=1}^{l-1}e_j)+(\sum\limits_{j=1}^{l-1}k_j)}[e_l+k_l+1])_q\\+S^{\nu}(k_1,\cdots,k_l-1,\cdots,k_p)(q^{n+(\sum\limits_{j=\tau_l+1}^{n}f_j)+(\sum\limits_{j=1}^{p}e_j)+(\sum\limits_{j=1}^{l}k_j)}[f_{\tau_l}-k_l+1]_q) \nonumber,
\end{align}
which is equivalent to the following (after cancelling a common factor)
\begin{align*}
    [e_l+1]_q \binom{e_l+1}{k_l}_q \binom{f_{\tau_l}}{k_l}_q=[e_l+k_l+1]_q \binom{e_l}{k_l}_q \binom{f_{\tau_l}}{k_l}_q\\+q^{e_l-k_l+1}[f_{\tau_l}-k_l+1]_q\binom{e_l}{k_l-1}_q \binom{f_{\tau_l}}{k_l-1}_q.
\end{align*}
And this can be checked by a direct computation. Now the left hand side of \eqref{ahffk} becomes \begin{align*}\label{eq6}
    &\sum\limits_{k_1,\cdots,k_p}S^{\nu}(k_1,\cdots,k_p)g^{(k_1,\cdots,k_p)}(\tau_l)\nonumber\\=&\sum\limits_{k_1,\cdots,k_p}S^{\nu}(k_1,\cdots,k_p))( q^{\tau_l+(\sum\limits_{j=1}^{l-1}e_j)+(\sum\limits_{j=1}^{l-1}k_j)}[e_l+k_l+1]_q+q^{n+(\sum\limits_{j=\tau_l+1}^{n}f_j)+(\sum\limits_{j=1}^{p}e_j)+(\sum\limits_{j=1}^{l}k_j)+1}[f_{\tau_l}-k_l]_q )
    \nonumber\\=&\sum\limits_{k_1,\cdots,k_p}S^{\nu}(k_1,\cdots,k_p)( q^{\tau_l+(\sum\limits_{j=1}^{l-1}e_j)+(\sum\limits_{j=1}^{l-1}k_j)}[e_l+k_l+1])_q\\&+\sum\limits_{k_1,\cdots,k_p}S^{\nu}(k_1,\cdots,k_l-1,\cdots,k_p)(q^{n+(\sum\limits_{j=\tau_l+1}^{n}f_j)+(\sum\limits_{j=1}^{p}e_j)+(\sum\limits_{j=1}^{l}k_j)}[f_{\tau_l}-k_l+1]_q ).\nonumber \\
    =&\sum\limits_{k_1,\cdots,k_p}(q^{\tau_l+(\sum\limits_{j=1}^{l-1}e_j)}[e_l+1]_q S^{\nu(\tau_l)}(k_1,\cdots,k_p)),
\end{align*}
where the last equality uses \eqref{tlqkf}. So this proves \eqref{ahffk}.

We will show \eqref{second} using an induction on $e_1$. The base case $e_1=0$ is same as \eqref{first}. Assume \eqref{second} holds for $\nu=((-1)^{e_1},\tau_2^{e_2},\cdots,\tau_p^{e_p})$. And let  $\nu'=(-1,\nu)$ and $\nu^{+}=\nu+1$, then we have
\begin{equation}\label{eq8}
   \nu'(i)=(-1,\nu(i)), \hspace{2mm} \nu^{+}(i)=\nu(i-1)+1.
\end{equation}
Writing \eqref{second} for $\nu$ and \eqref{first} for $\nu^{+}$, we have
\begin{align}\label{secondd}
    [n+b+1+\sum\limits_{i=1}^{p}e_i]_q M_{n}^{\nu}(b)=&(\sum\limits_{i=2}^{p}q^{(\tau_i+\sum\limits_{l=1}^{i-1}e_l)}[e_i+1]_q M_{n}^{\nu(\tau_i)}(b))+[e_1]_q M_{n}^{\nu(-1)}(b)\\&+\sum\limits_{i=1}^{p-1}(\sum\limits_{l=\tau_i+1}^{\tau_{i+1}-1}q^{(\sum\limits_{j=1}^{i}e_j)+l}M_{n}^{\nu(l)}(b))+\sum\limits_{l=\tau_p+1}^{n}q^{(\sum\limits_{j=1}^{p}e_j)+l} M_{n}^{\nu(l)}(b)\nonumber
    \end{align}
    \begin{align}
    &\label{seconddd}[n+b+1+\sum\limits_{i=1}^{p}e_i]_q M_{n+1}^{\nu^{+}}(b-1)\\=&\sum\limits_{i=2}^{p}q^{(\tau_i+\sum\limits_{l=1}^{i-1}e_l+1)}[e_i+1]_q M_{n+1}^{\nu^{+}(\tau_i+1)}(b-1))+[e_1+1]_q M_{n+1}^{\nu^{+}(0)}(b-1)\nonumber\\&+\sum\limits_{i=1}^{p-1}(\sum\limits_{l=\tau_i+1}^{\tau_{i+1}-1}q^{(\sum\limits_{j=1}^{i}e_j)+l+1}M_{n+1}^{\nu^{+}(l+1)}(b-1))+\sum\limits_{l=\tau_p+1}^{n}q^{(\sum\limits_{j=1}^{p}e_j)+l+1} M_{n+1}^{\nu^{+}(l+1)}(b-1)\nonumber .
\end{align}

Multiplying $q$ to \eqref{secondd}, multiplying $(q^{n+b+2+(\sum\limits_{i=1}^{p}e_i)}Y_{-1})$ to \eqref{seconddd} and adding these two we have 
\begin{align*}
    &q[n+b+1+\sum\limits_{i=1}^{p}e_i]_q (M_{n}^{\nu}(b)+q^{n+b+1+(\sum\limits_{i=1}^{p}e_i)}Y_{-1}M_{n+1}^{\nu^{+}}(b-1))\\=&\sum\limits_{i=2}^{p}q^{(\tau_i+\sum\limits_{l=1}^{i-1}e_l+1)}[e_i+1]_q (M_{n}^{\nu(\tau_i)}(b)+q^{n+b+2+(\sum\limits_{i=1}^{p}e_i)}Y_{-1}M_{n+1}^{\nu^{+}(\tau_i+1)}(b-1))\\&+q[e_1]_q (M_{n}^{\nu(-1)}(b)+q^{n+b+2+(\sum\limits_{i=1}^{p}e_i)}Y_{-1}M_{n+1}^{\nu^{+}(0)}(b-1))+q^{n+b+2+(\sum\limits_{i=1}^{p}e_i)}Y_{-1}M_{n+1}^{\nu^{+}(0)}(b-1)\\&+\sum\limits_{i=1}^{p-1}(\sum\limits_{l=\tau_i+1}^{\tau_{i+1}-1}q^{(\sum\limits_{j=1}^{i}e_j)+l+1}(M_{n}^{\nu(l)}(b)+q^{n+b+2+(\sum\limits_{i=1}^{p}e_i)}Y_{-1}M_{n+1}^{\nu^{+}(l+1)}(b-1))))\\&+\sum\limits_{l=\tau_p+1}^{n}q^{(\sum\limits_{j=1}^{p}e_j)+l+1} (M_{n}^{\nu(l)}(b)+q^{n+b+2+(\sum\limits_{i=1}^{p}e_i)}Y_{-1}M_{n+1}^{\nu^{+}(l+1)}(b-1)))\nonumber.
\end{align*}
By Lemma \ref{xx} and \eqref{eq8}, it becomes
\begin{align*}
    &q[n+b+1+\sum\limits_{i=1}^{p}e_i]_q M_{n}^{\nu'}(b)\\=&(\sum\limits_{i=2}^{p}q^{(\tau_i+\sum\limits_{l=1}^{i-1}e_l+1)}[e_i+1]_q M_{n}^{\nu'(\tau_i)}(b))+q[e_1]_q M_{n}^{\nu'(-1)}(b)+q^{n+b+2+(\sum\limits_{i=1}^{p}e_i)}Y_{-1}M_{n+1}^{\nu^{+}(0)}(b-1)\\&+\sum\limits_{i=1}^{p-1}(\sum\limits_{l=\tau_i+1}^{\tau_{i+1}-1}q^{(\sum\limits_{j=1}^{i}e_j)+l+1}M_{n}^{\nu'(l)}(b))+\sum\limits_{l=\tau_p+1}^{n}q^{(\sum\limits_{j=1}^{p}e_j)+l+1} M_{n}^{\nu'(l)}(b)\nonumber.
\end{align*}
Adding $M_{n}^{\nu'}(b)$($=M_{n}^{\nu(-1)}(b)$) to both sides gives \eqref{second} for $\nu'$.
\end{proof}
\subsection{Proof of Theorem \ref{main2}}\label{pt} For a weakly increasing composition $\mu=(\mu_1,\cdots,\mu_a)$, we define $X_{\mu}=\prod\limits_{i=1}^{a}X_{\mu_i+i-1}$. Then \eqref{main22} can be rephrased as follow 
\begin{equation}\label{ma}
    g_{n+k,n}=\sum\limits_{a+b=k}(\sum\limits_{\substack{\mu=(\mu_1,\cdots,\mu_a)\\0\leq\mu_1\leq \cdots\leq\mu_a\leq n}}X_{\mu}M_{n}^{\mu}(b)).
\end{equation}
Since we already know that Theorem \ref{main2} is true for $g_{k,0}$ (Example \ref{ex38}, Proposition \ref{prop1}), it suffices to prove that \eqref{ma} satisfies the recurrence relation
\begin{equation}\label{ff}g_{n+1+k,n+1}=g_{n+k,n}+(b_{n+k}) g_{n+k,n+1}-(\lambda_{n+k}) g_{n+k-1,n+1}.\end{equation}
We will show the identity
\begin{align}\label{dm}
    \sum\limits_{\substack{\mu=(\mu_1,\cdots,\mu_a)\\0\leq\mu_1\leq \cdots\leq\mu_a\leq n+1}}X_{\mu}M_{n+1}^{\mu}(b)=\sum\limits_{\substack{\mu=(\mu_1,\cdots,\mu_a)\\0\leq\mu_1\leq \cdots\leq\mu_a\leq n}}X_{\mu}M_{n}^{\mu}(b)+X_{n+a+b}(\sum\limits_{\substack{\mu=(\mu_1,\cdots,\mu_{a-1})\\0\leq\mu_1\leq \cdots\leq\mu_{a-1}\leq n+1}}X_{\mu}M_{n+1}^{\mu}(b))\\+Y_{n+a+b}(\sum\limits_{\substack{\mu=(\mu_1,\cdots,\mu_a)\\0\leq\mu_1\leq \cdots\leq\mu_a\leq n+1}}X_{\mu}M_{n+1}^{\mu}(b-1))-(\lambda_{n+a+b})(\sum\limits_{\substack{\mu=(\mu_1,\cdots,\mu_{a-1})\\0\leq\mu_1\leq \cdots\leq\mu_{a-1}\leq n+1}}X_{\mu}M_{n+1}^{\mu}(b-1)),\nonumber
\end{align}
which gives \eqref{ff} when summed over all possible $a$ and $b$ such that $a+b=k$. Using Lemma \ref{ana} and Lemma \ref{xx}, the identity \eqref{dm} becomes
\begin{align*}
    \leftrightarrow\nonumber \sum\limits_{\substack{\mu=(\mu_1,\cdots,\mu_a)\\0\leq\mu_1\leq \cdots\leq\mu_a\leq n+1}}X_{\mu}(M_{n+1}^{\mu}(b)-Y_{n+a+b}(M_{n+1}^{\mu}(b-1))=\sum\limits_{\substack{\mu=(\mu_1,\cdots,\mu_a)\\0\leq\mu_1\leq \cdots\leq\mu_a\leq n}}X_{\mu}M_{n}^{\mu}(b)\\\nonumber+X_{n+a+b}(\sum\limits_{\substack{\mu=(\mu_1,\cdots,\mu_{a-1})\\0\leq\mu_1\leq \cdots\leq\mu_{a-1}\leq n+1}}X_{\mu}(M_{n+1}^{\mu}(b)-Y_{n+a+b-1}(M_{n+1}^{\mu}(b))\\+(\alpha q^{n+a+b}Y_{-1})(\sum\limits_{\substack{\mu=(\mu_1,\cdots,\mu_{a-1})\\0\leq\mu_1\leq \cdots\leq\mu_{a-1}\leq n+1}}X_{\mu}M_{n+1}^{\mu}(b-1))\nonumber
    \end{align*}
    \begin{align*}
     \leftrightarrow\nonumber \nonumber\sum\limits_{\substack{\mu=(\mu_1,\cdots,\mu_a)\\0\leq\mu_1\leq \cdots\leq\mu_a\leq n+1}}X_{\mu}(M_{n}^{\mu-1}(b))=\sum\limits_{\substack{\mu=(\mu_1,\cdots,\mu_a)\\0\leq\mu_1\leq \cdots\leq\mu_a\leq n}}X_{\mu}M_{n}^{\mu}(b)\\+X_{n+a+b}(\sum\limits_{\substack{\mu=(\mu_1,\cdots,\mu_{a-1})\\0\leq\mu_1\leq \cdots\leq\mu_{a-1}\leq n+1}}X_{\mu}(M_{n}^{\mu-1}(b))+(\alpha q^{n+a+b}Y_{-1})(\sum\limits_{\substack{\mu=(\mu_1,\cdots,\mu_{a-1})\\0\leq\mu_1\leq \cdots\leq\mu_{a-1}\leq n+1}}X_{\mu}M_{n+1}^{\mu}(b-1))\nonumber
    \end{align*}
    \begin{align*}
    \leftrightarrow\nonumber \nonumber\sum\limits_{\substack{\mu=(\mu_1,\cdots,\mu_a)\\0\leq\mu_1\leq \cdots\leq\mu_a\leq n}}X_{\mu}(M_{n}^{\mu-1}(b))+\sum\limits_{\substack{\mu=(\mu_1,\cdots,\mu_{a-1},n+1)\\0\leq\mu_1\leq \cdots\leq\mu_{a-1}\leq n+1}}X_{\mu}(M_{n}^{\mu-1}(b))=\sum\limits_{\substack{\mu=(\mu_1,\cdots,\mu_a)\\0\leq\mu_1\leq \cdots\leq\mu_a\leq n}}X_{\mu}M_{n}^{\mu}(b)\\\nonumber+X_{n+a+b}(\sum\limits_{\substack{\mu=(\mu_1,\cdots,\mu_{a-1})\\0\leq\mu_1\leq \cdots\leq\mu_{a-1}\leq n+1}}X_{\mu}(M_{n}^{\mu-1}(b))+(\alpha q^{n+a+b}Y_{-1})(\sum\limits_{\substack{\mu=(\mu_1,\cdots,\mu_{a-1})\\0\leq\mu_1\leq \cdots\leq\mu_{a-1}\leq n+1}}X_{\mu}M_{n+1}^{\mu}(b-1))\nonumber\\
    \end{align*}
    \begin{align*}
    \leftrightarrow\nonumber\\
    \sum\limits_{\substack{\nu=(\nu_1,\cdots,\nu_{a-1})\\0\leq\nu_1\leq \cdots\leq\nu_{a-1}\leq n+1}}(X_{(\nu,n+1)}M_{n}^{(\nu,n+1)-1}(b)-X_{n+a+b}X_{\nu}M_{n}^{\nu-1}(b)-\alpha q^{n+a+b} Y_{-1} X_{\nu}M_{n+1}^{\nu}(b-1))\\
    =\sum\limits_{\substack{\mu=(\mu_1,\cdots,\mu_a)\\0\leq\mu_1\leq \cdots\leq\mu_a\leq n}}X_{\mu}(M_{n}^{\mu}(b)-M_{n}^{\mu-1}(b)).\nonumber
\end{align*}
\begin{align}\label{fff}
    \leftrightarrow\nonumber\\
    \sum\limits_{\substack{\nu=(\nu_1,\cdots,\nu_{a-1})\\0\leq\nu_1\leq \cdots\leq\nu_{a-1}\leq n+1}}(X_{(\nu,n+1)}M_{n}^{(\nu,n+1)-1}(b)-(X_{n+a+b}-X_0)X_{\nu}M_{n}^{\nu-1}(b)-X_0X_{\nu}M_{n}^{(-1,\nu-1)}(b))\\
    =\sum\limits_{\substack{\mu=(\mu_1,\cdots,\mu_a)\\0\leq\mu_1\leq \cdots\leq\mu_a\leq n}}X_{\mu}(M_{n}^{\mu}(b)-M_{n}^{\mu-1}(b)).\nonumber
\end{align}
We will show \eqref{fff} by showing its refinement (Proposition \ref{dx}).
\begin{Def}
For $\nu=(\tau_{1}^{e_1},\cdots,\tau_{p}^{e_p})$ such that $0\leq\tau_{1}<\cdots<\tau_{p}$ and $e_i>0$, we define 
\begin{align*}
\bar{\nu}(i) =
    \begin{cases*}
    (i,(\tau_{1}-1)^{e_1},\cdots,(\tau_{p}-1)^{e_p}) & if $-1\leq i\leq \tau_{1}-1$ \\
      (\tau_{1}^{e_1},\cdots,\tau_{l}^{e_l},i,(\tau_{l+1}-1)^{e_{l+1}},\cdots,(\tau_{p}-1)^{e_p}) & if $\tau_l\leq i\leq \tau_{l+1}-1$ \\
      (\tau_{1}^{e_1},\cdots,\tau_{p}^{e_p},i) & if $ \tau_{p}\leq i$
    \end{cases*}
\end{align*}
\end{Def}
\begin{Prop}\label{dx}
For $\nu=(\nu_1,\cdots,\nu_{a-1})$ such that $0\leq \nu_{1}\leq\cdots\leq \nu_{a-1}\leq n+1$, denoting $\nu^{-}=\nu-1$, we have
\begin{align}\label{dddd}
    X_{(\nu,n+1)}M_{n}^{(\nu,n+1)-1}(b)-(X_{n+a+b}-X_0)X_{\nu}M_{n}^{\nu-1}(b)-X_0X_{\nu}M_{n}^{(-1,\nu-1)}(b)\\
    =\sum\limits_{i=0}^{n} X_{\bar{v}(i)}(M_n^{\nu^{-}(i)}(b)-M_n^{\nu^{-}(i-1)}(b)).\nonumber
\end{align}
\end{Prop}
\begin{proof}
    Note that we have 
    \begin{equation*}
        \nu^{-}(n)=(\nu,n+1)-1,\hspace{2mm} \nu^{-}(-1)=(-1,\nu-1), \hspace{2mm} \bar{\nu}(n+1)=(\nu,n+1).
    \end{equation*}
    So the identity \eqref{dddd} becomes 
    \begin{equation}\label{dmdmd}
    (X_{n+a+b}-X_0)X_{\nu}M_{n}^{\nu^{-}}(b)
    =M_n^{\nu^{-}(-1)}(b)(X_{\bar{v}(0)}-X_{\nu}X_0)+\sum\limits_{i=0}^{n}M_n^{\nu^{-}(i)}(b)(X_{\bar{v}(i+1)}-X_{\bar{v}(i)}).
    \end{equation}
    Writing $\nu^{-}=(\tau_1^{e_1},\cdots,\tau_p^{e_p})$ such that $-1\leq\tau_1<\cdots<\tau_p$ and $e_i>0$, we have 
     \begin{align*}
&X_{\bar{v}(0)}-X_{v}X_0=
    \begin{cases*}
   0 & if $\tau_1>-1$ \\
       X_{\nu}(X_{e_1}-X_0) & if $\tau_1=-1$
    \end{cases*}\\
&X_{\bar{v}(i+1)}-X_{\bar{v}(i)}=
    \begin{cases*}
    X_{\nu}(X_{i+1}-X_i) & if $0\leq i\leq \tau_{1}-1$ \\
       X_{\nu}(X_{k+1}-X_{k}) & if $\tau_l+1\leq i\leq \tau_{l+1}-1$ \text{and} $k=(\sum\limits_{j=1}^{l}e_j)+i$\\
      X_{\nu}(X_{k+1}-X_k) & if $ \tau_{p}+1\leq i\leq n$ \text{and} $k=(\sum\limits_{j=1}^{p}e_j)+i$
      \\
      X_{\nu}(X_{k+e_l+1}-X_k) & if $ i=\tau_l$ \text{and} $k=(\sum\limits_{j=1}^{l-1}e_j)+\tau_l$
    \end{cases*}.
\end{align*}
Since $X_{k+i}-X_k=q^{k}[i]_q (X_1-X_0)$, dividing \eqref{dmdmd} with $X_\nu(X_1-X_0)$, we have \eqref{first} or \eqref{second} for $\nu^{-}$ which was proved in Lemma \ref{t}.
\end{proof}

\textit{Proof of the identity \eqref{fff}.} By Proposition \ref{dx}, the left hand side of \eqref{fff} becomes
\begin{equation}\label{final}\sum\limits_{\substack{\nu=(\nu_1,\cdots,\nu_{a-1})\\0\leq\nu_1\leq \cdots\leq\nu_{a-1}\leq n+1}}\left(\sum\limits_{i=0}^{n} X_{\bar{v}(i)}(M_n^{\nu^{-}(i)}(b)-M_n^{\nu^{-}(i-1)}(b))\right),\end{equation}where $\nu^{-}=\nu-1$. For $\mu=(\mu_1,\cdots,\mu_a)$ such that  $0\leq \mu_1 \leq \cdots \leq\mu_a\leq n$, the term $X_{\mu}$ appears on \eqref{final} when $\nu=(\mu_1,\cdots,\mu_{l-1},\mu_{l+1}+1,\cdots,\mu_a+1)$ with the coefficient  
$(M_{n}^{(\mu_1-1,\cdots,\mu_{l-1}-1,\mu_l,\cdots,\mu_a)}(b)-M_{n}^{(\mu_1-1,\cdots,\mu_{l}-1,\mu_{l+1},\cdots,\mu_a)}(b))$ for $l=1,\cdots,a$. So \eqref{final} becomes
\begin{align*}
    &\sum\limits_{\substack{\mu=(\mu_1,\cdots,\mu_a)\\0\leq\mu_1\leq \cdots\leq\mu_a\leq n}}X_{\mu}(\sum\limits_{l=1}^{a}(M_{n}^{(\mu_1-1,\cdots,\mu_{l-1}-1,\mu_l,\cdots,\mu_a)}(b)-M_{n}^{(\mu_1-1,\cdots,\mu_{l}-1,\mu_{l+1},\cdots,\mu_a)}(b)))\\
    &=\sum\limits_{\substack{\mu=(\mu_1,\cdots,\mu_a)\\0\leq\mu_1\leq \cdots\leq\mu_a\leq n}}X_{\mu}(M_{n}^{\mu}(b)-M_{n}^{\mu-1}(b)).
\end{align*}

\hspace{145mm}$\square$
\section{Proof of Theorem \ref{2pos}}
We prepare with lemmas to prove Theorem \ref{2pos}. Recall that we write $f_1\succeq f_2$ if $(f_1-f_2)$ is a polynomial with positive coefficients. 

\begin{Lemma}\label{11}
Let $f_1,\cdots,f_i$ and $h_1,\cdots,h_i$ be polynomials such that $f_j \succeq h_j$ for all $j$ and $h_1,\cdots,h_{i-1} \geq 0$. Then $f_1\cdots f_i \succeq h_1\cdots h_i$. 
\end{Lemma}
\begin{proof}
We have $f_1\cdots f_i -h_1\cdots h_i=(f_1\cdots f_i- h_1\cdots h_{i-1}f_i)+h_1\cdots h_{i-1}(f_i-h_i)$ and it is trivial to see
$(f_1\cdots f_i- h_1\cdots h_{i-1}f_i)\succeq 0$ and $h_1\cdots h_{i-1}(f_i-h_i)\succeq0$.
\end{proof}

\begin{Lemma}\label{45}
For weakly increasing compositions $\mu$ and $\nu$ such that $l_1=length(\mu) \geq l_2=length(\nu)$, $\mu_1\geq-1$, $\nu_1\geq0$ and $\mu_i \leq \nu_i$ for all possible $i$, we have $M_{n-l_2}^{\nu}(b) \succeq M_{n-l_1}^{\mu}(b)$ for all valid $n$ and $b$.
\end{Lemma}
\begin{proof}
It is enough to show $m_{n-l_2}^{\nu}(B) \preceq m_{n-l_1}^{\mu}(B)$ for all $B \subset \{0,\cdots,n+b-1\}$ with $|B|=b$. Let $B \cap\{n+b-l_2,\cdots,n+b-1\}=\{n+b-j_k\cdots<n+b-j_1\}$ and  $B \cap\{n+b-l_1,\cdots,n+b-l_2-1\}=\{n+b-j_{k+k'}\cdots<n+b-j_{k+1}\}$. Then we have 
\begin{align*}
    &m_{n-l_1}^{\mu}(B)=(\prod\limits_{i \in B \cap \{0,\cdots,n+b-l_1-1\}} Y_i)(\prod\limits_{i=1}^{k+k'} q^{(n+b-j_i)-B(\mu_{j_i}+i)}Y_{B(\mu_{j_i}+i)}) \\
    &m_{n-l_2}^{\nu}(B)=(\prod\limits_{i \in B \cap \{0,\cdots,n+b-l_1-1\}} Y_i)(\prod\limits_{i=k+1}^{k+k'} Y_{n+b-j_i})(\prod\limits_{i=1}^{k'} q^{(n+b-j_i)-B(\nu_{j_i}+i)}Y_{B(\nu_{j_i}+i)}).
\end{align*}
Since $\mu_{j_i} \leq \nu_{j_i}$ we have $B(\mu_{j_i}+i) \leq B(\nu_{j_i}+i)$, which implies $$q^{(n+b-j_i)-B(\mu_{j_i}+i)}Y_{B(\mu_{j_i}+i)} \preceq q^{(n+b-j_i)-B(\nu_{j_i}+i)}Y_{B(\nu_{j_i}+i)}$$ for $1\leq i \leq k$. We also have $$q^{(n+b-j_i)-B(\mu_{j_i}+i)}Y_{B(\mu_{j_i}+i)} \preceq Y_{n+b-j_i}$$
for $k+1\leq i \leq k+k'$. And every term is a polynomial with positive coefficients except for $q^{(n+b-1)-B(0)}Y_{B(0)}=q^{n+b}Y_{-1}$ when $(n+b-1)\in B$ and $\mu_1=-1$. So we have $m_{n-l_2}^{\nu}(B) \preceq m_{n-l_1}^{\mu}(B)$ for all $B \subset \{0,\cdots,n+b-1\}$ by Lemma \ref{11}. 
\end{proof}
\begin{Ex}
For $\mu=(-1,1,1)$, $\nu=(0,1)$ and $B=\{0,5,6,7\}$, we have
\begin{align*}
    &m_{1}^{\mu}(B)=m_{1}^{(-1,1,1)}(\{0,5,6,7\})=Y_0(qY_4)(q^3Y_3)(q^8Y_{-1})\\
    & m_{2}^{\nu}(B)=m_{2}^{(0,1)}(\{0,5,6,7\})=Y_0(Y_5)(q^3Y_3)(q^6Y_{1}).
\end{align*}
Since $qY_4 \preceq Y_5$, $ q^8Y_{-1}\preceq q^6Y_1$ and every term except $(q^8Y_{-1})$ is a polynomial with positive coefficients, we have 
$m_{1}^{\mu}(B) \preceq m_{2}^{\nu}(B)$.
\end{Ex}
\begin{Lemma}\label{22}
For $\mu=(\mu_1,\cdots,\mu_l)$ such that $0\leq\mu_1\leq\cdots\leq\mu_l\leq n$, we have
\begin{align} \label{key}
     M_{n}^{\mu}(b) \preceq  \sum\limits_{j=1}^{l}(\sum\limits_{\substack{0\leq \nu_{1}\leq\cdots\leq\nu_k \leq n+l\\ \nu_1=n+l+1-j}} (M_{n}^{(\mu_j,\cdots,\mu_l)}(b-k)\prod\limits_{i=1}^{k}Y_{\nu_i+b-k+i-1}))\\ +\sum\limits_{\substack{0\leq \nu_1\leq\cdots\leq\nu_k \leq n+l\\ \nu_1\leq n}} (M_{\nu_1}(b-k)\prod\limits_{i=1}^{k}Y_{\nu_i+b-k+i-1})\nonumber.
\end{align}
\end{Lemma}

\begin{proof}
    By Lemma \ref{ana}, we have  $M_{n}^{\mu}(b)=Y_{n+l+b-1}M_{n}^{\mu}(b-1)+M_{n-1}^{\mu-1}(b)$ and by Lemma \ref{22}, we have $M_{n-1}^{\mu-1}(b) \preceq M_{n}^{(\mu_2,\cdots,\mu_l)}(b)$ which gives 
    \begin{equation} \label{e}
        M_{n}^{\mu}(b)\preceq Y_{n+l+b-1}M_{n}^{\mu}(b-1)+M_{n}^{(\mu_2,\cdots,\mu_l)}(b).
    \end{equation}
    Applying \eqref{e} to $M_{n}^{(\mu_2,\cdots,\mu_l)}(b)$ on the right hand side, we have  
    \begin{equation}\label{f}
        M_{n}^{\mu}(b)\preceq Y_{n+l+b-1}M_{n}^{\mu}(b-1)+Y_{n+l+b-2}M_{n}^{(\mu_2,\cdots,\mu_l)}(b-1)+M_{n}^{(\mu_3,\cdots,\mu_l)}(b-1).
    \end{equation}
    Keeping this process, we have 
     \begin{align} \label{g}
        M_{n}^{\mu}(b)\preceq Y_{n+l+b-1}M_{n}^{\mu}(b-1)+Y_{n+l+b-2}M_{n}^{(\mu_2,\cdots,\mu_l)}(b-1)+\cdots \\ +Y_{n+b}M_{n}^{(\mu_l)}(b-1)+M_n(b) \nonumber. 
    \end{align}
    Since $M_n(b)$ is an elementary symmetric polynomial of degree $b$ with variables from $Y_0$ to $Y_{n+b-1}$, we have 
    \begin{equation} \label{h}
        M_n(b)=Y_{n+b-1} M_n(b-1)+Y_{n+b-2} M_{n-1}(b-1)+\cdots+Y_{b-1} M_0(b-1).
    \end{equation}
    Applying \eqref{h} to \eqref{g} gives
    \begin{align} \label{i}
        M_{n}^{\mu}(b)\preceq Y_{n+l+b-1}M_{n}^{\mu}(b-1)+Y_{n+l+b-2}M_{n}^{(\mu_2,\cdots,\mu_l)}(b-1)+\cdots \\ +Y_{n+b}M_{n}^{(\mu_l)}(b-1)+Y_{n+b-1} M_n(b-1)+Y_{n+b-2} M_{n-1}(b-1)+\cdots+Y_{b-1} M_0(b-1) \nonumber.
    \end{align}
    The inequality \eqref{i} corresponds to the case $k=1$ of \eqref{key}. Applying \eqref{h} and \eqref{i} to $M_{n}^{\mu}(b-1),\cdots,M_{n}^{(\mu_l)}(b-1),M_n(b-1),\cdots, M_0(b-1)$ on the right hand side of \eqref{i} gives $k=2$ of \eqref{key}. Subsequently applying \eqref{h} and \eqref{i} gives \eqref{key} for any $k$.  
\end{proof}
Note that we have $M_{n}^{(\mu_j,\cdots,\mu_l)}(b-k)\preceq M_{n}^{\mu}(b-k)$ for $1\leq j\leq l$ and $M_{j}(b-k)\preceq M_{n}^{\mu}(b-k)$ for $0\leq j\leq n$. Applying this to inequality \eqref{key} gives 
\begin{equation*}
 M_{n}^{\mu}(b) \preceq  M_{n}^{\mu}(b-k) (\sum\limits_{\substack{0\leq \nu_1\leq\cdots\leq\nu_k \leq n+l}} (\prod\limits_{i=1}^{k}Y_{\nu_i+b-k+i-1})),
\end{equation*}
and since $\sum\limits_{\substack{0\leq \nu_1\leq\cdots\leq\nu_k \preceq n+l}} (\prod\limits_{i=1}^{k}Y_{\nu_i+b-k+i-1}) \preceq M_{n+l+b-k}(k)$, we have 
\begin{equation}\label{kk}
 M_{n}^{\mu}(b) \preceq  M_{n}^{\mu}(b-k)M_{n+l+b-k}(k).
\end{equation}
\textit{Proof of Theorem \ref{2pos}.} By Theorem \ref{main2}, we have
\begin{equation*}
    g_{n+a+b,n}=\sum_{c=0}^{a+b}(\sum\limits_{\substack{\mu=(\mu_1,\cdots,\mu_c) \\ 0\leq\mu_1\leq\cdots\leq\mu_c \leq n}} X_{\mu}M_{n}^{\mu}(a+b-c)).
\end{equation*}
When $c\geq a$, let $\mu'=(\mu_1,\cdots,\mu_a)$ and $\nu=(\mu_{a+1}+a,\cdots,\mu_c+a)$ then $X_{\mu}=X_{\mu'}X_{\nu}$ and $M_{n}^{\mu}(a+b-c) \preceq M_{n+a}^{\nu}(a+b-c)$ by Lemma \ref{45}. So we have 
\begin{equation}\label{e1}
     X_{\mu}M_{n}^{\mu}(a+b-c)) \preceq X_{\mu'}(X_{\nu}M_{n+a}^{\nu}(a+b-c)).
\end{equation}
When $c<a$, by \eqref{kk}, we have 
\begin{equation*}
    M_{n}^{\mu}(a+b-c) \preceq  M_{n}^{\mu}(a-c)M_{n+a}^{\mu}(b),
\end{equation*}
which gives 
\begin{equation}\label{e2}
     X_{\mu}M_{n}^{\mu}(a+b-c)) \preceq (X_{\mu}M_{n}^{\mu}(a-c))M_{n+a}(b).
\end{equation}
Terms on the right hand sides of \eqref{e1} and \eqref{e2} appear in $g_{n+a+b,n+a}g_{n+a,n}$ and they do not overlap. So summing up \eqref{e1} and \eqref{e2} for all possible $c$ and $\mu$ gives 
\begin{equation*}
     g_{n+a+b,n}\preceq g_{n+a+b,n+a}g_{n+a,n}.
\end{equation*}
\hspace{145mm}$\square$

\printbibliography
\end{document}